\documentclass[11pt,a4paper,twoside]{amsart}
\usepackage{amsmath,amsfonts,amsthm,amsopn,color,amssymb,enumitem}
\usepackage{palatino}
\usepackage{graphicx}
\usepackage[colorlinks=true]{hyperref}
\hypersetup{urlcolor=blue, citecolor=red, linkcolor=blue}

\usepackage{relsize}
\usepackage{esint}
\usepackage{verbatim}
\usepackage{mathrsfs}

\newcommand{\cS}{\mathscr{S}}
\newcommand{\e}{\varepsilon}
\newcommand{\emb}{\hookrightarrow}

\newcommand{\R}{\mathbb{R}}
\newcommand{\N}{\mathbb{N}}

\newcommand{\de}{\partial}
\newcommand{\weakto}{\rightharpoonup}
\newcommand{\oR}{\overline{R}}

\DeclareMathOperator{\dv}{div}
\DeclareMathOperator{\supp}{supp}
\DeclareMathOperator{\dist}{dist}

\DeclareMathOperator{\dive}{div}

\newcommand{\Sph}{{\mathbb{S}}}

\newcommand{\D}{{\mathfrak{D}}}

\newcommand{\func}{\overline{\varphi}_{\beta,D}}
\newcommand{\dsh}{{2^\sharp}}
\newcommand{\dst}{{2^*}}
\newcommand{\hsp}{\hspace{0.2cm}}

\newcommand*{\modgrad}[1]{\left\vert \nabla #1\right\vert^2}
\newcommand*{\inv}[1]{\frac{1}{#1}}
\newcommand*{\abs}[1]{\left\vert #1\right\vert}
\newcommand*{\norm}[1]{\left\Vert #1\right\Vert}

\newcommand{\beq }{\begin{equation}}
	\newcommand{\eeq }{\end{equation}}

\newtheorem{theorem}{Theorem}[section]
\newtheorem{lemma}[theorem]{Lemma}

\newtheorem{definition}[theorem]{Definition}
\newtheorem{proposition}[theorem]{Proposition}
\newtheorem{remark}[theorem]{Remark}
\newtheorem{corollary}[theorem]{Corollary}

\title[Conformal metrics with prescribed scalar and mean curvature]{Conformal metrics with prescribed scalar and mean curvature}

\author{Sergio Cruz-Bl\'{a}zquez}
\address{Sergio Cruz-Bl\'{a}zquez, Scuola Normale Superiore, Piazza dei Cavalieri 7, 56126 Pisa, Italy.}
\email{sergio.cruzblazquez@sns.it }

\author{Andrea Malchiodi}
\address{Andrea Malchiodi, Scuola Normale Superiore, Piazza dei Cavalieri 7, 56126 Pisa, Italy.}
\email{andrea.malchiodi@sns.it }

\author{David Ruiz}
\address{David Ruiz, Departamento de Análisis Matemático, University of Granada, Avda. Fuentenueva s/n, 18071 Granada, Spain.}
\email{daruiz@ugr.es}

\thanks{S. C-B. is supported by the Marie Sklodowska-Curie fellowship of the Istituto Nazionale di Alta Matematica "Francesco Severi" number 713485. D. R. has been supported by the FEDER-MINECO Grant PGC2018-096422-B-100 and by J. Andalucia (FQM-116). A.M. has been supported by the project {\em Geometric Variational Problems} and {\em Finanziamento a supporto della ricerca di base} from Scuola Normale Superiore and by MIUR Bando PRIN 2015 2015KB9WPT$_{001}$.  He is also member of GNAMPA as part of INdAM. Part of this work was completed while 
he was visiting F.I.M. at ETH in Z\"urich, to which he is grateful}



\begin{document}

	\begin{abstract}
		We consider the  case with boundary of the classical Kazdan-Warner problem in dimension 
		greater or equal than three, i.e. the 
		prescription of scalar and boundary mean curvatures via conformal deformations 
		of the metric. We deal in particular with negative scalar curvature and 
		boundary mean curvature of arbitrary sign, which to our knowledge has not been treated
		in the literature. We employ a variational approach to prove new existence results, 
		especially in three dimensions. One of the principal issues for this problem is to obtain compactness 
		properties, due to the fact that {\em bubbling} may occur with profiles of 
		hyperbolic balls or horospheres, and hence one may lose either pointwise estimates on the 
		conformal factor or the total conformal volume. We can sometimes prevent them using 
		integral estimates, Pohozaev identities and domain-variations of different types. 
	\end{abstract}

	\maketitle

	\section{Introduction}\setcounter{equation}{0}

	The problem of prescribing the scalar curvature of a manifold via conformal transformations 
	was initially proposed by Kazdan-Warner in the 70's and received a lot of attention. 
	We will not give here a complete account  on the numerous contributions to it, 
	but  will  refer instead to Chapter 6 in the book \cite{aubin-book} for a 
	general presentation of it and some collection of results. 
	
	\
	
	We are interested  in the case of compact manifolds with boundary and of 
	dimension $n \geq 3$: the natural counterpart in this case is to conformally prescribe 
	both the scalar curvature in the interior part and the mean curvature at the boundary, 
	which results in a nonlinear equation of critical type for the conformal factor. 
	
	More precisely, lat $(M^n, \partial M, g)$ be a $n$-dimensional Riemannian manifold  with 
	scalar curvature $S=S_g$ and boundary mean curvature $h=h_g$. Considering a
	conformal metric $\tilde{g} = u^{\frac{4}{n-2}} g$, with $u : \overline{M} \to \R$ a 
	smooth positive function, the new curvatures $\tilde{S} = S_{\tilde{g}}$ and $\tilde{h} = h_{\tilde{g}}$ 
	are determined by the following formulas, see for instance \cite{cherrier}:
	\begin{align} 
		\left\lbrace \begin{array}{ll}-\frac{4(n-1)}{n-2}\Delta_{g} u + S u = \tilde{S} u^{\frac{n+2}{n-2}} &\mbox{on } M, \\[0.2cm]
			\frac{2}{n-2}\frac{\partial u}{\partial\eta}+ h u = \tilde{h} u^{\frac{n}{n-2}} &\mbox{on } \partial M.\end{array}\right. \label{eq:change}
	\end{align}

	Here $\Delta_g$ stands for the Laplace-Beltrami operator associated to the metric $g$ and $\eta$ is the unit outer normal to $\partial M$.
	
	The question then one asks is the following: given smooth functions $K$ and $H$, \textit{is there any conformal metric $g$ such that the scalar curvature of $M$ and the mean curvature of $\partial M$ with respect to $g$ are $K$ and $H$?}
	
	\

	While the boundary problem \eqref{eq:change} has been in general less investigated than the 
	closed case, still there are results worth mentioning. 	In \cite{cherrier} the regularity of solutions to \eqref{eq:change} was shown, and 
	some first existence results were proven up to Lagrange multipliers. 
	An interesting particular case 
	is the counterpart of the \emph{Yamabe problem}, consisting in prescribing 
	constant functions $H$ and $K$. This study was initiated by Escobar in \cite{escobar Annals}, 
	\cite{escobar}, \cite{escobar Indiana}, with later contributions in \cite{HanLi}, \cite{HanLi2}. 
	For more recent progress in this direction we refer to \cite{Mayer-Ndiaye} and  
	its comprehensive list of references. 
	
	\

	The case of variable functions has been treated in special situations: see e.g. 
	\cite{BEO02}, \cite{BEO05}, \cite{li-domain} concerning minimal boundaries (i.e. $H = 0$), and  
	\cite{Abde}, \cite{CXY},  \cite{DMO04},  \cite{XZ16} for the case $K = 0$ (scalar-flat metrics).

	On the variable-curvature case, we mention the paper \cite{AML}, 
	containing perturbative results about nearly-constant functions on the unit 
	ball of $\R^n$, the case of the 
	three-dimensional half sphere in \cite{djadlimalchiodiahmedou}, with $K > 0$, 
	and \cite{CHS18} for negative curvatures (more precisely, for manifolds with 
	boundary of \emph{negative Yamabe invariant})  where the equations are 
	solved up to some Lagrange multipliers.

	\

	Our goal here is to consider variable functions $K < 0$ and $H$ of arbitrary sign, and to give results about existence of solutions and bubbling behavior. In fact we obtain some counterparts of the results in \cite{LSMR}, which 
	was dealing with the two-dimensional case for domains with positive genus (see also 
	\cite{sergio}, \cite{pistoia}, \cite{JLSMR},  for the case of the disk). As we will see, 
	there are many differences when the dimension is higher than two.

	To state our theorems, we first reduce the problem to a simpler situation using  a result by Escobar \cite[Lemma 1.1]{escobar}. 
	This implies that, without losing generality, via an initial conformal change one can start with $h_g = 0$ and $S_g = S$ not changing sign: 
	this will be always \underline{assumed understood} in the rest of the paper. 
	
		In view of \eqref{eq:change}, we are led to find \emph{positive} solutions of the boundary value problem:
	\begin{align}
	\left\lbrace \begin{array}{ll}-\frac{4(n-1)}{n-2}\Delta_{g} u + S u = K u^{\frac{n+2}{n-2}} &\mbox{on } M, \\[0.2cm]
	\frac{2}{n-2}\frac{\partial u}{\partial\eta} = H u^{\frac{n}{n-2}} &\mbox{on } \partial M.\end{array}\right. \label{eq2}
	\end{align}
	
	\
	
	Solutions of \eqref{eq2} can be obtained as critical points of the following energy functional, defined on $H^1(M)$:
	\begin{equation}\label{eq3}
		I(u)=\frac{2(n-1)}{n-2}\int_M \vert \nabla u\vert^2+\frac{1}{2}\int_M S u^2 -\frac{1}{\dst}\int_M K u^{\dst}-(n-2)\int_{\partial M}Hu^{\dsh},
	\end{equation} 
	where $\dst = \frac{2n}{n-2}$ and $\dsh = \frac{2(n-1)}{n-2}$ are the critical Sobolev exponent for $M$ and the critical trace embedding exponent for $\partial M$, respectively. Here and below we  omit the volume and surface elements in the integrals. 
	As written before, we will assume that $K<0$, so that the third term in the right-hand side of \eqref{eq3} is positive. The interaction between this term and the boundary critical terms will be crucial for the behavior of the energy functional.
	
	\

	Via a trace inequality we show that, in fact, the nature of the functional is determined by a quotient of the prescribed curvatures at the boundary, and it also allows us to compare both terms. For convenience, define the
	\emph{scaling invariant} function $\D_n:\de M\to \R$ by 
	\begin{equation} \label{def Dn}
		\mathfrak{D}_n(x) = \sqrt{n(n-1)}\frac{H(x)}{\sqrt{\abs{K(x)}}}. 
	\end{equation}
	Depending on whether  $\D_n$ is strictly less than $1$ or not, we find ourselves in two completely different scenarios. We notice that boundaries of geodesic spheres in hyperbolic spaces 
	satisfy $\D_n > 1$, while $\D_n = 1$ at boundaries of \emph{horospheres}. Therefore, 
	when $\D_n \geq 1$, there might be blow-ups for  \eqref{eq2} with 
	such  profiles, see more precise later comments in this direction. 
	
	\
	
	Assuming that $\D_n(x)<1$ for every $x\in\partial M$, it turns out that $K$ \emph{shadows} $H$, and the corresponding positive term in $I$ dominates the one at the boundary involving $H$. The result is that the functional becomes coercive and a global minimizer can be found. 
	
	Our first result concerns the case when the Escobar metric satisfies $S<0$, and compared to 
	\cite{CHS18}, \cite{cherrier}  we can solve the original geometric problem without any 
	extra Lagrange multiplier. 
	\begin{theorem}\label{minimum1}
		Suppose $K<0$, and that $\D_n$ as in \eqref{def Dn} verifies $\mathfrak{D}_n < 1$ everywhere on 
		$\partial M$. Then, if $S < 0$, \eqref{eq2} admits a solution.
	\end{theorem}
	If $S=0,$ extra hypotheses are needed to rule out the possibility of the minimizer beging identically zero, so the solution we  obtain is geometrically admisible.
	\begin{theorem}\label{minimum2}
		Suppose $K<0$ on $M$ and  $\mathfrak{D}_n < 1$ on $\partial M$. Then, if $S = 0$ and $\int_{\partial M}H>0$, \eqref{eq2} admits a solution.
	\end{theorem}
	On the other hand, if there exists $p\in\partial M$ such that $\D_n(p)>1,$ we can  construct a sequence of functions $u_i$ such that the energy $I(u_i)$ tends to $-\infty.$ While this prevents the existence of minimizers,  in three dimension we can  use the Mountain-pass Theorem to obtain a solution for \eqref{eq2}. 
	\begin{theorem}\label{minmax} Let $n=3$, assume that $S=0$, $K<0$ and that $H$ is such that 
		\begin{enumerate}
			\item $\int_{\de M}H<0$,
			\item $\D_n(\overline{p})>1$ for some $\overline{p}\in \de M$, and
			\item $1$ is a regular value for $\D_n$.
		\end{enumerate}
		Then, \eqref{eq2} admits a positive solution.
	\end{theorem}

	We will explain below why we have a dimensional restriction   in Theorem \ref{minmax}, 
	while giving an outline of the proof. To prove the existence of Min-Max solutions it is necessary to show that Palais-Smale sequences of approximate solutions to \eqref{eq2} converge. For doing this, two main difficulties occur: first, one needs to show 
	their boundedness in norm, which is unclear in our case due to the triple homogeneity of the 
	Euler-Lagrange functional. Second, because of the presence of critical exponents in \eqref{eq2}, 
	even bounded Palais-Smale sequences may not converge.

	\
	
	To deal with the first issue we make use of  \emph{Struwe's monotonicity trick}, see \cite{str88}, 
	consisting in perturbing the problem by a parameter with respect to which the energy is monotone, together 
	with a subcritical approximation to guarantee compactness of 
	Palais-Smale sequences. We consider for these reasons  the following situation.

	Let $(K_i)_i$  be a sequence of regular functions defined on $M$ such that $K_i\to K$ in $C^2(\overline{M})$ and  let $(H_i)_i$ be a sequence of smooth functions on $\partial M$ such that $H_i\to H$ in $C^2(\partial M)$. Assuming that $K<0$, we consider positive solutions 
	$(u_i)_i$ to the perturbed problem 
	\begin{align}\label{eq4}
		\left\lbrace \begin{array}{ll}
			-4\frac{n-1}{n-2}\Delta_g u_i+S u_i = K_i {u_i}^{p_i} & \mbox{ on } M, \\[0.2cm]
			\frac{2}{n-2} \frac{\partial u_i}{\partial \eta} = H_i {u_i}^{\frac{p_i+1}{2}} & \mbox{ on } \partial M, 
		\end{array} \right.
	\end{align}
namely critical points of the energy functional:
	\begin{equation}\label{eqIi}
	I_i(u)=\frac{2(n-1)}{n-2}\int_M \vert \nabla u\vert^2+\frac{1}{2}\int_M S u^2 -\frac{1}{p_i+1}\int_M K_i  |u|^{p_i+1}-4\frac{n-1}{p_i+3}\int_{\partial M}H_i |u|^{\frac{p_i+3}{2}},
\end{equation} 
	with $p_i\nearrow \frac{n+2}{n-2}$. The question is then whether such solutions could be 
	uniformly bounded from above, in which case they would converge to a solution 
	of the original problem \eqref{eq2}. 

Assuming the contrary, consider a sequence $(u_i)$  of solutions of \eqref{eq4}, and define its \emph{singular set} as
	\begin{equation*}
		\mathscr{S} = \{p\in \overline{M}:\exists x_i\to p \mbox{ such that }u_i(x_i) \mbox{ is unbounded}\}.
	\end{equation*}
	In this regard, we have the following compactness result, which is interesting in its own right.
	\begin{theorem}\label{compactness} Let $(u_i)$ be a sequence of solutions of \eqref{eq4}, and $\mathscr{S}$ be the associated singular set. Then
		\begin{enumerate}
			\item $\mathscr{S}\subset \{p\in \partial M: \D_n(p)\geq 1\}$.
		\end{enumerate}
		Therefore, we can write $\mathscr{S}=\mathscr{S}_0 \sqcup \mathscr{S}_1$, with $\mathscr{S}_1 = \mathscr{S}\cap \{\D_n>1\}$ and $\mathscr{S}_0=\mathscr{S}\cap \{\D_n=1\}$. In dimension $n=3$, we have further:
		\begin{enumerate}
			\item[(2.1)] $\mathscr{S}_1$ is a finite collection of points.
			\item[(2.2)] If $S\leq 0$, then $\mathscr{S}_1=\emptyset$.
			\item[(2.3)] If $I_i(u_i)$ is uniformly bounded and  $1$ is a regular value of $\D_n$, then $\mathscr{S}_0=\emptyset$.
		\end{enumerate}
	\end{theorem}
	
	The above result gives a description of two types blow-up points, gathered in the sets $\mathscr{S}_0$ and $\mathscr{S}_1$. Blow-up profiles are solutions of the following problem in the half-space
	\begin{equation}\label{eq5}
		\left\lbrace\begin{array}{ll}\frac{-4(n-1)}{n-2}\Delta v =K(p)v^{\frac{n+2}{n-2}} & \mbox{on }\R^n_+ \\\frac{2}{n-2}\frac{\partial v}{\partial\eta}=H(p)v^{\frac{n}{n-2}} & \mbox{ on } \partial \R^n_+\end{array}\right.,
	\end{equation}
	where $p\in S$. Solutions to \eqref{eq5} were classified in \cite{ChipotFilaShafrir} (see also \cite{li-zhu}) as follows: 
	\begin{itemize}
		\item[\textcolor{blue}{$\star$}] If $\D_n(p)<1$, there are no solutions.
		\item[\textcolor{blue}{$\star$}] If $\D_n(p)=1$, the only solutions are \emph{$1-$dimensional} and given by:
		\begin{equation*}
			v(x)=v_\alpha(x):=\left(\frac{2}{\sqrt{n(n-2)}}x_n+\alpha\right)^{-\frac{n-2}{2}},
		\end{equation*}
		for any $\alpha > 0.$
		\item[\textcolor{blue}{$\star$}] If $\D_n(p)>1$, the solutions are called \emph{bubbles} and  given by
		\begin{equation*}
			v(x)=b_\beta(x):=\frac{\left(n(n-2)\right)^{\frac{n-2}{4}}\beta^{\frac{n-2}{2}}}{\left(\abs{x-x_0(\beta)}^2-\beta^2\right)^{\frac{n-2}{2}}},
		\end{equation*}
		with $x_0(\beta)=-\D_n(p)\beta e_n \in \R^n$,  $\beta > 0$.
	\end{itemize}
	
	We would like to emphasize that the singular set $\mathscr{S}$ can be infinite, contrarily to what happens in the case without boundary, at least in low dimensions. The development of a blow-up analysis in a situation of an infinite number of blow-up points is one of the main goals of this paper. Moreover, both types of blow-up behavior are indeed possible; the one at points of $\mathscr{S}_1$ can be produced by the invariance of the problem under conformal maps of the disk, in analogy with what happens in the closed case. But in this framework we can have blow-up around sets infinite sets $\mathscr{S}_0$. In Section \ref{s:prel} we show an easy example showing that this is possible.

	Compared to the two-dimensional case studied in \cite{LSMR} we have more rigidity in the classification, since for the half-plane other solutions are generated by meromorphic functions, see \cite{mira-galvez}. On the other hand, in the two-dimensional case one can make use \emph{complex-analytic tools} which are not availabe in higher dimensions.

	To deal with loss of compactness at points where $\D_n>1$, we perform a precise study on the blow-up behavior, 
	showing that, \underline{in dimension $n = 3$} they are \emph{isolated and simple} and therefore they form a finite collection (see also as for \cite{djadlimalchiodiahmedou} in this regard). Once this is proved, we are able to control solutions 
	also away from such points, dismissing this kind of blow-up by some integral estimates which hold 
	true when $S \leq 0$, see Proposition \ref{dtb1}.

	On the other hand, around blow-up points with $\D_n = 1$, the terms $\int_M \modgrad{u_i}$, $\int_M \abs{K_i}{u_i}^{p_i+1}$ and $\int_{\de M}H_i{u_i}^{\frac{p_i+3}{2}}$ diverge. Assuming \emph{boundedness of the energies} $I_i(u_i)$, (which is natural for min-max sequences) we can prove that they converge weakly towards the same measure after proper normalization. Using then a \emph{domain-variation 
		technique} we show that at such blow-up 
	points the gradient of $\D_n$ at $\{ \D_n = 1 \}$ must vanish, 
	contradicting our assumption on the regularity of this level. Compared to a similar step 
	in \cite{LSMR} for the two-dimensional case, we have to choose arbitrary deformations 
	tangent to $\partial M$.

	\
	
	The plan of the paper is the following. In Section \ref{s:prel} we collect some preliminary facts, including 
	the classification of limiting profiles and some integral identities deriving from domain variations. 
	Section \ref{s:variational} is devoted to the geometry of the Euler-Lagrange functional I, derived from a 
	trace inequality, which can give rise to either a minimization or a mountain-pass scheme depending 
	on the values of the prescribed curvatures. Section \ref{s:blowup} is concerned with some general properties of 
	blowing-up solutions, including the existence of limiting profiles for any point in $\cS$. 
	In Section \ref{s:i-s-bu} we prove decay properties of solutions presenting isolated-simple blow-ups, 
	while in Section \ref{s:n=3} we specialize to the three-dimensional case and prove that all blow-ups 
	are of this type. In Section \ref{s:pf-thm} we then prove Theorem \ref{minmax} ruling out all blow-up possilibities, 
	including the set $\{\D_n = 1\}$, via integral estimates. In the Appendix we finally collect some test function 
	estimates. 
	
	\
	
	\begin{center}
		\textbf{Notation}
	\end{center}

	For the sake of simplicity, we sometimes use the notation:
	
	$$ C_n = \frac{4(n-1)}{n-2}.$$
	
	For $X= M$ or $X = \R^n_+$, and given a domain $\Omega$ in $X$, we use the symbol $\partial_0 \Omega$ when  
	referring to $\bar{\Omega} \cap \partial X$, while $\partial^+ \Omega$ 
	will stand for the set $\partial \Omega \cap X$. Moreover, we will write $\Omega_+ = \Omega \cap X$. When working in geodesic normal 
	coordinates centered at some point of $M$, we will often use the 
	symbol $|x-y|$ for the Riemannian distance of two points with coordinates $x$ and $y$ 
	respectively.

	For functions defined on $\partial M$ we define their derivatives with a superscript $T$. For instance, $\nabla K$ denotes the gradient of $K$ and $\nabla^T K$ its tangential gradient on $\partial M$. Moreover, $\eta$ denotes always the unit outer normal vector at the corresponding boundary. 
	
	For notational convenience, we will often omit volume or surface elements in 
	integrals and we will use the symbol $\int$ to denote boundary integration: we will also denote by $C$ positive constants that might vary 
	from one formula to another, or also within the same one.

	\section{Preliminaries} \setcounter{equation}{0} \label{s:prel}
	
	In this section we introduce some preliminary results that will be of use in the rest of the paper.
	%
	%
	%
	We start by recalling the following Lemma by Escobar:
	\begin{lemma} (\cite{escobar}, Lemma 1.1.) If $(M^n, g)$ is a compact Riemannian manifold with boundary
		and $n\geq 3$, there exists a conformal metric to $g$ whose scalar curvature does not change sign and the boundary is minimal.
	\end{lemma}
	So, without loss of generality, we will from now on assume that $\partial M$ is minimal and $S$ does not change sign in equation \eqref{eq:change}. Then we are led to equation \eqref{eq2}.
	
	\subsection{The limit problem and its solutions}
	
	When performing blow-up analysis, one is usually concerned with certain limit problems after a proper rescaling of solutions. In this case, we are interested in solutions with constant curvatures in the half-space:
	
	\begin{equation}\label{lp1}
		\left\lbrace \begin{array}{ll}
			-\Delta v = \frac{n-2}{4(n-1)}K(0)v^{\frac{n+2}{n-2}} & \mbox{on } \R^n_+ \\ \frac{\partial v}{\partial \eta } = \frac{n-2}{2}H(0)v^{\frac{n}{n-2}}, & \mbox{on } \partial\R^n_+.
		\end{array} \right.
	\end{equation}
	
	The next result was proved  in \cite{ChipotFilaShafrir}.
	
	\begin{proposition}\label{lp2} The following assertions hold true:
		\begin{enumerate}
			\item If $\D_n(0)<1$ then \eqref{lp1} does not admit any solution.
			\item If $\D_n(0)=1$, the only solutions are $1-$dimensional and given by:
			\begin{equation}\label{lp3}
				v(x)=v_\alpha(x):=\left(\frac{2}{\sqrt{n(n-2)}}x_n+\alpha\right)^{-\frac{n-2}{2}},
			\end{equation}
			for any $\alpha > 0.$
			\item If $\D_n(0)>1$, the solutions  (called \textit{bubbles})  are given by
			\begin{equation}\label{lp4}
				v(x)=b_\beta(x):=\frac{\left(n(n-2)\right)^{\frac{n-2}{4}}\beta^{\frac{n-2}{2}}}{\left(\abs{x-x_0(\beta)}^2-\beta^2\right)^{\frac{n-2}{2}}},
			\end{equation}
			with $x_0(\beta)=-\D_n(0)\beta e_n \in \R^n$, for  $\beta > 0$ arbitrary. In this case, we highlight the following assymptotic behavior:
			\begin{equation} \label{lp18}
				\lim_{\abs{x}\to+\infty}\abs{x}^{n-2}b_{\beta_0}(x) = \left(n(n-2)\right)^{\frac{n-2}{2}}{\beta_0}^{\frac{n-2}{2}}
			\end{equation}
			for any fixed $\beta_0>0$.
		\end{enumerate}
	\end{proposition}

	\subsection{Domain-variations}
	
	In order to derive global properties of blowing-up solutions to our problem, we will make use of domain-variations, 
	with calculations that are gathered in this subsection. For variation vector fields of radial type these coincide with the classical 
	Pohozaev identity, but we will need more general ones in Section \ref{s:pf-thm}. We start with some definitions that will be useful here and in further sections of the paper.
	
	\begin{definition} \label{def} 
		
		Given a point $p \in \partial M$, a function $u: M \to \R$ and a vector field $F:M \to TM$, we define $B_p(u,F): \partial M \to \R$,
	\begin{equation} B_p(u,F)=C_n |\cdot -p| \Big ( \left( F \cdot \eta\right)^2-\frac{1}{2} |F|^2 \Big )+2(n-1) u F \cdot \eta.
	\end{equation}
		
If the point $p$ is clear from the notation we will drop the subscript and write simply $B(u,F)$.
	\end{definition}

	\begin{lemma}\label{pi1} Given regular functions $\mathfrak{f},\mathfrak{g}:M\to \R$ and $\mathfrak{h}:\partial M\to \R$, and exponents $0\leq p\leq \frac{n+2}{n-2}$, $0\leq q\leq \frac{n}{n-2}$, consider a positive solution  $u \in C^2(\overline{M})$  of
		\begin{equation}\label{pi2}
				-C_n \Delta u + \mathfrak{g} u = \mathfrak{f} u^p \qquad \mbox{on}\hsp M, 
		\end{equation}
		with $C_n=\frac{4(n-1)}{n-2}$, and $F:M\to TM$ any smooth vector field. Then,
		\begin{align}\label{pi3}
			C_n\int_M DF(\nabla u,\nabla u)-\frac{C_n}{2}\int_M\modgrad{u}\dive{F}-\frac{1}{p+1}\int_M \mathfrak{f} F\cdot\nabla(u^{p+1}) \nonumber\\ +\frac{1}{2}\int_M \mathfrak{g} F\cdot\nabla(u^2) = C_n\int_{\partial M}(\nabla u\cdot F)\frac{\partial u}{\partial \eta} - \frac{C_n}{2}\int_{\partial M}\modgrad{u}F\cdot\eta,
		\end{align}
		where $DF\left(\nabla u,\nabla u\right):=\sum_{k,j=1}^n \nabla_j F^k u_k u^j.$ 
	\end{lemma}
	\begin{proof}
		Let us introduce the vector field $\mathbb{Y}=\left(\nabla u\cdot F\right)\nabla u-\frac{1}{2}\modgrad{u}F$: a direct computation shows that
		\begin{align}
			\dive \mathbb Y = \sum_{j=1}^n\nabla_j \mathbb{Y}^j = \sum_{k,j=1}^n \nabla_j F^k u_k u^j+\left(\nabla u\cdot F\right)\Delta u - \frac{1}{2}\modgrad{u}\dive F. \label{pi11}
		\end{align}
		Finally, multiply \eqref{pi2} by $\nabla u\cdot F$ and integrate by parts, using \eqref{pi11}.
	\end{proof}

	\begin{corollary}\label{pi4} Consider the domain $\Omega = B(0,r)^+\subset \R^n_+$, and let $u$ be as in Lemma \ref{pi1}. Define:
		\begin{align}\label{pi8}
			\mathfrak{P}_\Omega(u)=&\frac{1}{p+1}\int_\Omega u^{p+1}X\cdot\nabla \mathfrak{f} + \left(\frac{n}{p+1}-\frac{n-2}{2}\right)\int_\Omega \mathfrak{f} u^{p+1}\nonumber\\&-\frac{1}{2}\int_\Omega u^2 X\cdot \nabla \mathfrak{g} -\int_\Omega \mathfrak{g} u^2 +\frac{r}{2}\int_{\partial^+\Omega} \mathfrak{g} u^2-\frac{r}{p+1}\int_{\partial^+\Omega} \mathfrak{f} u^{p+1}.\end{align}
		Then, if $u$ solves also $	\frac{2}{n-2}\frac{\partial u}{\partial \eta} = \mathfrak{h} u^q$ on $\partial M$, 
		we have that 
		\begin{align}\label{pi9}
			\mathfrak{P}_\Omega(u)=&\int_{\partial^+ \Omega} B(u,\nabla u) + 2(n-1)\left(\frac{n-2}{2}-\frac{n-1}{q+1}\right)\int_{\partial_0\Omega}\mathfrak{h}u^{q+1} \nonumber\\&+ \frac{2(n-1)}{q+1}\int_{\partial(\partial_0\Omega)} \mathfrak{h}u^{q+1}(X\cdot\nu)  -\frac{2(n-1)}{q+1}\int_{\partial_0\Omega}u^{q+1}(\nabla \mathfrak{h}\cdot X) ,
		\end{align}

	\end{corollary}
	\begin{proof} We first apply Lemma \ref{pi1} with $F=X$, taking into account that $\dive X = n$ and $DX(\nabla u,\nabla u)=\modgrad{u}$. We obtain:
		\begin{align}
			-2(n-1)\int_\Omega\modgrad{u}&-\frac{1}{p+1}\int_\Omega \mathfrak{f} X\cdot\nabla(u^{p+1})+\frac{1}{2}\int_\Omega \mathfrak{g}X\cdot\nabla(u^2) \nonumber\\\label{pi12} &=C_n\int_{\partial\Omega}(\nabla u\cdot X)(\nabla u\cdot\eta)-\frac{C_n}{2}\int_{\partial\Omega}\modgrad{u}(X\cdot \eta).
		\end{align}
		Moreover, using the Divergence Theorem,
		\begin{align}
			\int_\Omega \mathfrak{f}X\cdot\nabla(u^{p+1})&=\int_{\partial \Omega}\mathfrak{f}u^{p+1}X\cdot \eta - \int_\Omega u^{p+1}X\cdot \nabla \mathfrak{f} - n\int_\Omega \mathfrak{f}u^{p+1}; \label{pi13}\\
			\int_\Omega gX\cdot\nabla(u^{2})&=\int_{\partial \Omega} \mathfrak{g}u^{2}X\cdot \eta - \int_\Omega u^{2}X\cdot \nabla \mathfrak{g} - n\int_\Omega \mathfrak{g}u^{p+1}. \label{pi14}
		\end{align}
		Multiplying \eqref{pi2} by $u$ and integrating by parts we can relate the last two terms, namely 
		\begin{equation}\label{pi15}
			\frac{n-2}{2}\int_\Omega \mathfrak{f}u^{p+1}-\frac{n-2}{2}\int_\Omega \mathfrak{g}u^2 = 2(n-1)\int_\Omega\modgrad{u}-2(n-1)\int_{\partial\Omega}u\frac{\partial u}{\partial\eta}.
		\end{equation}
		Combining \eqref{pi13} and \eqref{pi14} with \eqref{pi15} and pugging them into \eqref{pi12}, we obtain:
		\begin{align}\frac{-1}{p+1}\int_{\partial \Omega}\mathfrak{f}u^{p+1}X\cdot \eta + \frac{1}{p+1}\int_\Omega u^{p+1}X\cdot \nabla \mathfrak{f} -\frac{1}{2}\int_\Omega u^2 X\cdot\nabla \mathfrak{g} \nonumber\\+\frac{1}{2}\int_{\partial \Omega} \mathfrak{g}u^2 X\cdot \eta-\int_\Omega \mathfrak{g}u^2 +\left(\frac{n}{p+1}-\frac{n-2}{2}\right)\int_\Omega \mathfrak{f}u^{p+1}\nonumber \\ = C_n\int_{\partial\Omega} (\nabla u\cdot X)\frac{\partial u}{\partial \eta}-\frac{C_n}{2}\int_{\partial \Omega}\modgrad{u}(X\cdot\eta)+2(n-1)\int_{\partial \Omega}u\frac{\partial u}{\partial \eta}.\label{pi16}
		\end{align}
		Taking into account that the exterior normal vector to $\partial\Omega$ satisfies
		\begin{align*}
			\eta(x)=\left\lbrace\begin{array}{ll} \frac{x}{r} & \mbox{if}\hsp x\in \partial^+\Omega, \\ -e_n & \mbox{if}\hsp x\in\partial_0\Omega, \end{array}\right.
		\end{align*}
		we proceed to study the right-hand side of \eqref{pi16}. Integrating by parts we have
		\begin{align*}
			&\int_{\partial_0\Omega}\frac{\partial u}{\partial\eta} (X\cdot\nabla u)= \frac{n-2}{2(q+1)}\int_{\partial_0\Omega} \mathfrak{h} X\cdot\nabla(u^{q+1}) \nonumber\\&=
			\frac{n-2}{2(q+1)}\left(\int_{\partial(\partial_0\Omega)} \mathfrak{h} u^{q+1}(X\cdot\nu)-\int_{\partial_0\Omega}u^{q+1}(X\cdot\nabla \mathfrak{h})-(n-1)\int_{\partial_0\Omega}\mathfrak{h}u^{q+1}\right).
		\end{align*}
		Therefore,
		\begin{align}
			C_n \int_{\partial\Omega} (X\cdot\nabla u)\frac{\partial u}{\partial \eta} &= C_n\int_{\partial^+\Omega}r\left(\frac{\partial u}{\partial \eta}\right)^2+C_n\int_{\partial_0\Omega}(X\cdot\nabla u)\frac{\partial u}{\partial \eta} \nonumber\\ &= C_n\int_{\partial^+\Omega}r\left(\frac{\partial u}{\partial \eta}\right)^2 + \frac{2(n-1)}{q+1}\int_{\partial(\partial_0\Omega)}\mathfrak{h}u^{q+1}(X\cdot\nu) \nonumber\\ &-2(n-1)\int_{\partial_0\Omega}u^{q+1}(X\cdot\nabla \mathfrak{h})-\frac{2(n-1)^2}{q+1}\int_{\partial_0\Omega}\mathfrak{h}u^{q+1}.\label{pi17}
		\end{align}
		In addition,
		\begin{equation}
			\frac{-C_n}{2}\int_{\partial\Omega}\modgrad{u}(X\cdot \eta) = -\frac{C_n}{2}r\int_{\partial^+\Omega}\modgrad{u} \label{pi18}\end{equation} 	
		\begin{equation} 2(n-1)\int_{\partial\Omega}u\frac{\partial u}{\partial\eta} = 2(n-1)\int_{\partial^+\Omega}u\frac{\partial u}{\partial\eta}+(n-1)(n-2)\int_{\partial_0\Omega}\mathfrak{h}u^{q+1}.\label{pi19}
		\end{equation} 
		Finally, notice also that, by \eqref{pi2}
		\begin{equation*}
			\frac{n-2}{2}\int_{\partial_0\Omega}\mathfrak{h}u^{q+1} = \int_{\partial_0\Omega}u\frac{\partial u}{\partial \eta}.
		\end{equation*}
		Identity \eqref{pi9} is a consequence of \eqref{pi16}, \eqref{pi17}, \eqref{pi18} and \eqref{pi19}.
	\end{proof}
	As a final goal of this section, we study a particular case of Corollary \ref{pi4} that will be useful later on.
	
	\
	
	For any fixed constant $a>0$, define the function $G(x)=a\abs{x}^{2-n}$: direct computations show that
	\begin{equation*}
		B(G,\nabla G) = -2(n-1)\left(\frac{(n-2)a^2}{|x|^{2n-3}}-2\frac{(n-2)a^2}{|x|^{2n-3}}+\frac{(n-2)a^2}{|x|^{2n-3}}\right)=0.
	\end{equation*}
	Furthermore, we have the following result:
	\begin{proposition}\label{pri11} Define the function $h:B(r)_+\to \R$ as 
		\begin{equation*}
			h(x) = G(x)+b(x),
		\end{equation*}
		for any function $b\in C^1(\overline{B(r)_+})$. Then,
		\begin{equation}\label{pi10}
			\lim_{r\to 0}\int_{\partial^+\Omega}B(h,\nabla h) = -(n-1)(n-2)\omega_{n-1}a \, b(0).
		\end{equation}
		\begin{proof}
			The fact that $B(G,\nabla G)=0$ implies 
			\begin{equation*}
				B(h,\nabla h) = B_{r}(b,\nabla b)-\frac{6a(n-1)}{r^{n-1}}(X\cdot\nabla b)-\frac{2a(n-1)(n-2)}{r^{n-1}}b(x).
			\end{equation*}
			Integrating on $\partial^+B_r$, and taking into account that $\abs{\partial^+ B_r}=\frac{1}{2}\omega_{n-1}r^{n-1}$, we obtain:
			\begin{align*}
				\int_{\partial^+B_r}B(h,\nabla h) &= \int_{\partial^+B_r}B(b,\nabla b) - 3a(n-1)\omega_{n-1}\fint_{\partial^+B_r} X\cdot \nabla b \\ &-a(n-1)(n-2)\omega_{n-1}\fint_{\partial^+B_r}b.
			\end{align*}
			We conclude by taking the limit  $r\to 0$.
		\end{proof}
	\end{proposition}

	\subsection{An explicit example of blow-up with infinite singular set} 
	
	Here we show that the cardinality of the singular set $\cS_0 = \cS \cap \{\mathfrak{D}_n=1 \} $ can be infinite. Indeed, consider $\rho>1$ and the function $u_{\rho}: B(0,1) \subset \R^n \to \R$ as:
	
	$$ u_\rho(x)= \Big (\frac{2 \rho}{\rho^2- |x|^2}  \Big )^{\frac{n-2}{2}}.$$
	
	It is easy to check that $u_\rho$ solves:
	
	\begin{align}
		\left\lbrace \begin{array}{ll}-\frac{4(n-1)}{n-2}\Delta u = - n(n-1) u^{\frac{n+2}{n-2}} &\mbox{on } B(0,1), \\[0.2cm]
			\frac{2}{n-2}\frac{\partial u}{\partial\eta}+ u = H_\rho u^{\frac{n}{n-2}} &\mbox{on } \partial B(0,1),\end{array}\right. 
	\end{align}
	where $H_\rho= \frac{\rho^2+1}{2\rho}$. As $\rho \to 1$, $H_\rho \to 1$ but $K= -n(n-1)$. In this case, the function $u_\rho$ diverges on the whole boundary $\partial B(0,1) = \cS_0$.

	Obseve that the function $u$ gives rise to a model of the Hyperbolic space in a ball of radius $\rho >1$, and hence $H_\rho$ is nothing but the mean curvature of the sphere of Euclidean radius 1 in such a model.

	\section{The variational study of the energy functional} \setcounter{equation}{0} \label{s:variational}
	
	In this section we analyze  the geometric properties of the energy functional in \eqref{eq3}. This study will readily imply the proof of Theorems \ref{minimum1} and \ref{minimum2}. Moreover, we will show that under the assumptions of Theorem \ref{minmax}, $I$ satisfies the  hypotheses of the mountain pass lemma. However, the proof of Theorem \ref{minmax} will require the compactness result of Theorem \ref{compactness}, which will be proved further on.

	\medskip We begin with the following inequality,  showing that the nature of the functional is ruled by the interaction between its critical terms. This will allow us to prove Theorems \ref{minimum1} and \ref{minimum2}.
	
	\begin{proposition} \label{traceineq} For every $\e>0$ there exists $C>0$ such that
		
		\begin{align}\label{ti2}
			\int_{\partial M}H |u|^\dsh \leq (\bar{D}+ \e) \left(\frac{2(n-1)}{(n-2)^2}\int_M \modgrad{u}+\frac{1}{2n}\int_M\vert K \vert |u|^\dst \right) + C \int_{M} |u|^{\dsh},
		\end{align}
		where $\bar{D}= \max_{x \in \partial M}\{0, \D_n(x) \}$.
	\end{proposition}
	
	\begin{proof}
		
		Take a partition of  unity $\{\phi_j\}_{j=1}^m$  on $M$,  a vector field $N\in\mathfrak{X}(M)$ on $M$ with $\vert N\vert \leq 1$ and such that $N=\eta$ on $\partial M$. Firstly we see that, for every $1\leq j\leq m$ and $u\in H^1(M)$
		
		\begin{equation*}
			\int_{\partial M} \phi_j |u|^\dsh = \int_{\partial M}\phi_j |u|^\dsh N\cdot\eta.
		\end{equation*}
		
		Thus, by the Divergence Theorem
		\begin{align}\label{ti1}
			\int_{\partial M} \phi_j |u|^\dsh &= \int_M \{\phi_j \dv{N}+\nabla \phi_j N\}|u|^\dsh + \frac{2(n-1)}{n-2}\int_M \phi_j \nabla u \cdot N |u|^{\frac{2}{n-2}} u\nonumber \\ &\leq C \int_M |u|^\dsh + \frac{2(n-1)}{n-2}\int_M \phi_j |\nabla u| \cdot |u|^{\frac{n}{n-2}}.
		\end{align}
		
		Consider $\D_n$ as defined in \eqref{def Dn}, and let $H_j = \max\{H(x), x \in \supp\ \phi_j\}$, $|K|_j = \min\{|K(x)|, x \in \supp\ \phi_j\}$. Then, summing \eqref{ti1} on $j$ and assuming that the supports of the $\phi_j$'s are sufficiently small, we have:
		\begin{align*}
			(n-2)\int_{\partial M}H|u|^\dsh = (n-2)\sum_{j=1}^m \int_{\partial M}\phi_j H |u|^\dsh \\ \leq (n-2)\sum_{j=1}^m H_j \int_{\partial M}\phi_j |u|^\dsh \leq C(n-2) \Vert H\Vert_\infty \int_M |u|^\dsh \\ +2(n-1)\left(\sum_{j=1}^m \frac{H_j}{\sqrt{|K|_j}}\right) \int_M \vert \nabla u\vert |u|^{\frac{n}{n-2}}\sqrt{\vert K \vert} \\ \leq C' \int_M |u|^\dsh + 2\sqrt{\frac{n-1}{n}}(\bar{D}+ \varepsilon) \int_M \vert \nabla u\vert |u|^{\frac{n}{n-2}}\sqrt{\vert K \vert} \\ \leq C'\int_M |u|^\dsh + 2\sqrt{\frac{n-1}{n}}(\bar{D} + \e) \left(\frac{\lambda}{2}\int_M \modgrad{u} + \frac{1}{2\lambda} \int_M |u|^\dst \vert K\vert\right),
		\end{align*}
		for every $\lambda>0$. 
		Choosing $\lambda = \frac{2\sqrt{n(n-1)}}{n-2}$, and renaming $\e$ properly, we conclude. 
	\end{proof}
	
	\subsection{Proof of Theorems \ref{minimum1} and \ref{minimum2}}
	
	We mainly rely on the inequality proven in Proposition \ref{traceineq}. If $\bar{D} <1$, we can use \eqref{ti2} taking into account that by H\"{o}lder's inequality	
	$$ \int_M |u|^\dsh \leq \delta \int_{M} |u|^{2^*} + C.$$
	Then, for sufficiently small $\delta>0$,
	\begin{align}\label{c1}
		I(u)&\geq \delta \int_M \modgrad{u}+\frac{1}{2}\int_M Su^2 +\delta \int_M\vert K \vert \, |u|^\dst-C.
	\end{align}

	Hence, if we take a sequence $(u_i)$ in $H^1(M)$ such that $\Vert u_i\Vert_{H^1}\to +\infty$, then either $\Vert \nabla u_i\Vert_{L^2}$ or $\Vert u_i\Vert_{L^2}$ must tend to $+\infty$, which implies by \eqref{c1} that $I(u_i)\to +\infty$. The consequence is that $I$ is coercive.
	
	Now, let us show that a global minimizer for $I$ can always be found. Let us consider
	\begin{equation*}
		\alpha = \inf\left\lbrace I(u): u\in H^1(M) \right\rbrace, 
	\end{equation*}
	and a minimizing sequence $(u_i)$ in $H^1(M)$ such that $I(u_i)\to \alpha$. Coercivity implies that the sequence $(u_i)$ is bounded in  $H^1(M)$. Thus, up to a subsequence, $(u_i)\weakto u$ in $H^1(M)$.
	
	\
	
	Using Brezis-Lieb's result in \cite{BrezisLieb}, we decompose $I(u_i)$ as $I(u_i)=I(u)+I(u_i-u)+o_i(1)$ and study the second term in the right-hand side. Using the trace inequality \eqref{ti2} for $u_i-u$ and the compactness of the embedding $H^1(M) \hookrightarrow L^{\dsh}(M)$, we obtain:
	
	\begin{align*}
		I(u_i-u) &\geq \delta \left\lbrace\frac{2(n-1)}{n-2}\int_M \modgrad{(u_i-u)}+\frac{n-2}{2n}\int_M \abs{K}\abs{u_i-u}^\dst\right\rbrace \\ &+o_i(1). 
	\end{align*}
	Therefore,
	\begin{equation*}
		I(u_i)\geq I(u)+o_i(1),
	\end{equation*}
	and it suffices to take limits to see that $u$ is a minimizer for $I$ in $H^1(M)$. Again, by the relation $I(u)=I(\abs{u})$, we can assume that the minimizer is non-negative. In order to conclude, we need to ensure that $u>0$. 
	
	\
	
	\textit{Case 1:} $S=0$ and $\int_{\partial M}H>0$. Firstly, we check that $0\in H^1(M)$ is not a global minimum for $I$, so $u$ must be positive somewhere. Take $\e >0$ and consider
	\begin{equation*}
		I(\e) = \frac{n-2}{2n}\e^\dst\int_M \vert K \vert -(n-2)\e^{\dsh}\int_{\partial M} H,
	\end{equation*}
	being the second the leading term as $\e$ approaches zero since $\dsh < \dst$. Thus, there exists $\e_0>0$ such that $I(\e)<0$ for all $0<\e<\e_0$, and consequently $\inf I<0$. We can conclude that the minimizer $u$ is not identically zero. Observe now that $|u|$ is also a minimizer, so we can assume that $u \geq 0$.
	
	\
	
	In order to show that $u>0$ in $M$, we start recalling that it solves the problem
	\begin{align}\label{rbm1}
		\left\lbrace \begin{array}{ll}
			-4\frac{n-1}{n-2}\Delta_g u = K u  ^\frac{4}{n-2}u & \mbox{ on } M, \\[0.2cm]
			\frac{2}{n-2} \frac{\partial u}{\partial \eta} = H u^{\frac{2}{n-2}}u & \mbox{ on } \partial M.
		\end{array} \right.
	\end{align}
	Apply now the maximum principle (\cite[Ch. 3]{GilbargTrudinger}) to the elliptic operator
	\begin{equation*}
		L(\phi)=\frac{4(n-1)}{n-2}\Delta_g \phi + K \vert u\vert^{\frac{4}{n-2}}\phi.
	\end{equation*}
	
	By Hopf's Strong Maximum Principle (\cite[Th. 3.5]{GilbargTrudinger}), if there exists a point $x\in \mbox{Int }M$ with $u(x)=0,$ then $u$ achieves its minimum in $\mbox{Int }M$ and thus $u\equiv 0$, which is impossible.
	
	\
	
	Moreover, if there would exist  $x_0\in \partial M$ with $u(x_0)=0$, using the second equation in \eqref{rbm1} we would get that $\frac{\partial u}{\partial \eta}(x_0)=0$, which is a contradiction to Hopf's Lemma. 
	\
	
	\textit{Case 2:} $S<0$. Analogously, we can prove that $u\not\equiv 0$ by showing that $0$ is not a global minimum. To see this, take $\e>0$ and consider
	\begin{equation*}
		I(\e) = \frac{\e^\dst}{\dst}\int_M \vert K\vert - \frac{\e^2}{2}\int_M \vert S\vert - \e^\dsh(n-2)\int_{\partial M}H.
	\end{equation*}
	In this case, the leading term as $\e\to0$ is $-\frac{\e^2}{2}\int_M \vert S\vert<0$, so, again, there exists $\e_0>0$ such that $I(\e)<I(0)=0$ for all $0<\e<\e_0$. To see that $u>0$ one can follow the same strategy as in the first case.

	
	\subsection{A first step for proving Theorem \ref{minmax}}
	
	This subsection is devoted to the proof of the following result:
	
	\begin{proposition} \label{perturbed} Assume $S=0$, $K<0$ and that $H$ is taken so that
		\begin{enumerate}
			\item $\int_{\de M}H<0$,
			\item $\D_n(\overline{p})>1$ for some $\overline{p}\in \de M$.
		\end{enumerate}
		Consider a sequence of exponents $p_i\nearrow \frac{n+2}{n-2}$. Then there exist $\kappa_i \to 1$ and  solutions $u_i$ of the perturbed problem:
		\begin{equation}\label{ps1}
			\left\lbrace\begin{array}{ll}\frac{-4(n-1)}{n-2}\Delta u = Ku^{p_i} & \mbox{in}\hsp M,\\\frac{2}{n-2}\frac{\de u}{\de \eta}=\kappa_i  Hu^{\frac{p_i+1}{2}}&\mbox{on}\hsp \de M. \end{array}\right.
		\end{equation}
		Moreover the solutions have bounded energy, that is, $I_i(u_i)$ is bounded, with
		
		\begin{align}
			I_i(u)=\frac{2(n-1)}{n-2}\int_M\modgrad{u}+\frac{1}{p_i+1}\int_M \abs{K} |u|^{p_i+1} -\kappa_i\frac{4(n-1)}{p_i+3}\int_{\partial M}H |u|^{\frac{p_i+3}{2}} \label{ps-1}. 
		\end{align}
		
	\end{proposition}

	Proposition \ref{perturbed} is a first step in the proof of Theorem \ref{minmax} and gives existence of solutions of aproximating problems.

	First, we shall prove that under the assumptions of Proposition \ref{perturbed}, the energy functional $I$ has a mountain pass geometry.

	\begin{lemma}\label{ps-2} Assume that $K<0$, $S=0$ and $\int_{\partial M}H<0$. Then there exists $\e>0$ such that for any $u \in H^1(M)$, $\|u\| = \e$, we have that $I(u) > \delta >0$, where $\delta$ is independent of $u$. 
	\end{lemma}
	\begin{proof} 
		
		We write $u = \tilde{u} + \bar{u}$, where 
		$$ \bar{u} \in \R, \ \qquad 
		\int_{M} \tilde{u}=0.$$
		
		Then, we can write:
		
		$$I(u) \geq \frac{2(n-1)}{n-2}\int_M \vert \nabla \tilde{u} \vert^2 -(n-2)\int_{\partial M}H|u|^{\dsh} $$\begin{equation} \label{la} = \frac{2(n-1)}{n-2}\int_M \vert \nabla \tilde{u} \vert^2 -(n-2) {|\bar{u}|}^{\dsh} \int_{\partial M}H - (n-2) \int_{\partial M}H (|u|^{\dsh} - |\bar{u}|^{\dsh}).\end{equation}
		
		Now we estimate the last term in the right hand side as:
		
		$$ \left | \int_{\partial M} H (|\bar{u} + \tilde{u}|^{\dsh} - |\bar{u}|^{\dsh}) \right | \leq \|H\|_{L^{\infty}} \int_{\partial M} \dsh |\tilde{u}|(|\tilde{u}|^{\dsh-1}+ |\bar{u}|^{\dsh-1}) $$$$\leq C \int_{\partial M} |\tilde{u}|^{\dsh} + \int_{\partial M} |\bar{u}| \, |\tilde{u}|^{\dsh -1}.$$
		
		By the trace and Sobolev's inequalities,
		
		$$ \int_{\partial M} |\tilde{u}|^{\dsh} \leq C  \left ( \int_{M} |\nabla \tilde{u}|^2 \right )^{\dsh/2}.$$
		
		Moreover, by combining Young's inequality with the trace and Sobolev's inequality, we have:
		
		$$ \int_{\partial M} |\bar{u}| \, |\tilde{u}|^{\dsh -1} \leq \gamma \int_{\partial M} |\bar{u}|^{\dsh} + C |\tilde{u}|^{\dsh} \leq \gamma |\partial M| |\bar{u}|^{\dsh} + C_\gamma  \left ( \int_{M} |\nabla \tilde{u}|^2 \right )^{\dsh/2}, $$
		for some $\gamma>0 $ to be chosen. Coming back to \eqref{la}, we obtain that for certain constants $c_0>0$, $C>0$, one has 	
		$$ I(u) \geq c_0 \| \tilde{u}\|^2 + c_0 |\bar{u}|^{\dsh} - \gamma |\partial M| |\bar{u}|^{\dsh} - C  \| \tilde{u}\|^{\dsh}.$$
		
		By choosing $\gamma$ sufficiently small, we conclude the proof.
	\end{proof}	
	
	\medskip 
	
	In next lemma we show that under the hypotheses of Theorem \ref{minmax}, the energy functional $I$ is not bounded from below.
	
	\begin{lemma}\label{tf15} If there exists $p\in\partial M$ such that $\D_n(p)>1$, then one can find a sequence of functions $(\varphi_k)$ in $H^1(M)$ such that $\varphi_k>0$, $I(\varphi_k)\to -\infty$ as $k\to +\infty$.
	\end{lemma}
	
	The proof of this lemma is postponed to the appendix.
	
	\begin{proof}[Proof of Proposition \ref{perturbed}]
		
		As it can be seen, the previous two lemmas imply that the functional $I$ has a mountain-pass geometry.  However, there are several obstacles in order to prove the existence of a solution. The first one is that we do not know whether  Palais-Smale's sequences are bounded or not. But, even if they are, we cannot guarantee compactness because of the lack of a compact embedding in the critical Sobolev inequalities. We can bypass these difficulties by considering the perturbed problems \eqref{ps1}.

		If $\kappa_i$ is close to $1$ and $p_i$ is close enough to $\frac{n+2}{n-2}$, the functional $I_i$ also satisfies the geometric assumptions of the mountain-pass lemma. Let us fix $i \in \N$: by Struwe's Monotonicity trick, there exists a bounded Palais-Smale sequence $u_k^i$ for $I_i$ for some $\kappa_i$, with $|\kappa_i - 1 | < \frac{1}{i}$.
		
		In order to conclude the proof of Proposition \ref{perturbed}, we  intend to find a positive solution as a strong limit of $u_k^i$ as $k\to \infty$. We summarize the available information about the sequence $(u_k^i)_k$:
		\begin{enumerate}
			\item $\norm{u_k^i}_{H^1(M)}$ is uniformly bounded in $k$. Therefore, there exists a weak limit $u_k^i\weakto u^i$ in $H^1(M)$.
			\item $I_i(u_k^i)\to c_i$ as $k\to +\infty$, for some positive constant $c_i$ (which is, actually, bounded also in $i$). Then,
			\begin{align}\label{ps2}
				\frac{2(n-1)}{n-2}\int_M\modgrad{u_k^i} + \frac{1}{p_i+1}\int_M |K|\, |u_k^i|^{p_i+1} \nonumber\\-\kappa_i\frac{4(n-1)}{p_i+3}\int_{\partial M} H |u_k^i|^{\frac{p_i+3}{2}} = c_i+o_k(1).
			\end{align}
			\item $I(u_k^i)'\to 0$ in $H^{-1}(M)$, so $I'(u_k^i)[v]=o_k(1)$ for every $v\in H^1(M)$.
		\end{enumerate}
		Using the compactness of the embedding $H^1(M)\emb L^p(M)$ for $1\leq p<\dst$ as well as that of the trace inequality $H^1(M)\emb L^q(\de M)$ for $1\leq q<\dsh$, we obtain
		\begin{align}
			\int_M \abs{K}|u_k^i|^{p_i+1}\to \int_M\abs{K}|u^i|^{p_i+1},\; \hsp \int_{\de M} H |u_k^i|^{\frac{p_i+1}{2}}\to \int_{\de M} H |u^i|^{\frac{p_i+1}{2}},
		\end{align}
		as $k\to +\infty$. By testing $I_i'$ on $u_k^i - u^i$, we also find that:
				$$ I_i'(u_k^i)(u_k^i - u^i) = \frac{4(n-1)}{n-2} \Big ( \int_{M} | \nabla u_k^i|^2 - \nabla u_k^i\cdot \nabla u^i \Big ) + o_k(1).$$
		
		This implies the strong convergence $u_k^i  \to u^i$ as $k \to +\infty$. Hence $u^i$ is a nontrivial solution of the problem:
		
		$$\left\lbrace\begin{array}{ll}\frac{-4(n-1)}{n-2}\Delta u^i = K|u^i|^{p_i-1}u^i & \mbox{in}\hsp M.\\\frac{2}{n-2}\frac{\de u^i}{\de \eta}=\kappa_i H |u^i|^{\frac{p_i-1}{2}} u^i&\mbox{on}\hsp \de M.\end{array}\right.$$
		
A priori, such solutions could be sign changing. However, let us depart from a function $\psi>0$ such that $I^i(\psi)<0$ (given by Lemma \ref{tf15}), and define the continuous curve $\gamma:[0,1] \to H^1(M)$, $\gamma(t)= t \psi$. 

As is well-known, the (PS) sequence can be taken very close to the family of curves given by deformations of $\gamma$ under the gradient flow of $I^i$ (see for instance \cite{ghoussoub}). We can now use the fact that the gradient flow of $I^i$ leaves invariant the cone of nonnegative functions to conclude that the limit of the (PS) sequence $u^i$ is nonnegative. This is rather standard, see for instance \cite[Lemma 4.1, (a)]{bartsch}. By using the maximum principle as previously we conclude that $u_i >0$.
	\end{proof}

	\section{Generalities on the singular set }\label{s:blowup}\setcounter{equation}{0}
	
	In this section we prove (1) in Theorem \ref{compactness}. The tricky point here is that we cannot assume that around points in $\mathscr{S}_0$ there are local maxima of $u_i$. We perform then a blow-up analysis around a suitably defined sequence of points, which are chosen thanks to \emph{Ekeland's variational principle}.

	Let $K_i: M \to \R$ and $H_i: \partial M \to \R$ be sequences of regular functions  such that $K_i\to K$ and $H_i\to H$ in the $C^2$ sense. Assume that $K<0$ and that $S$ has constant sign. For a sequence $(u_i)_i$ of positive solutions to
	\begin{align}\label{tfgc1}
		\left\lbrace \begin{array}{ll}
			-4\frac{n-1}{n-2}\Delta_g u_i+S u_i = K_i {u_i}^{p_i} & \mbox{ on } M, \\[0.2cm]
			\frac{2}{n-2} \frac{\partial u_i}{\partial \eta} = H_i {u_i}^{\frac{p_i+1}{2}} & \mbox{ on } \partial M,
		\end{array} \right.
	\end{align}
	with $p_i\nearrow \frac{n+2}{n-2}$, we recall the definition of the singular set
	\begin{equation} \label{tfgc2}
		\mathscr{S} = \left\lbrace p\in \overline{M}:\:\exists (x_i)\to p \mbox{ such that } u_i(x_i)\to+\infty\right\rbrace.
	\end{equation}
	
	\begin{proposition} \label{Ekeland} $\mathscr{S}\subset \{p\in \partial M: \D_n\geq 1\}$. \end{proposition}

	The idea is to make a suitable rescaling and pass to a problem in a half-space, whose solutions have been classified (\cite{ChipotFilaShafrir}). In general, this analysis is made around points of local maximum for the sequence $u_i$. However, in our situation we cannot guarantee their existence, as for example $u_i$ might be monotone in some direction when 
		restricted to a portion of the boundary. We bypass this obstacle by choosing good sequences, even if they are not local maxima, by means of  Ekeland's variational principle. For convenience we state it below, see e.g. Chapter I in \cite{struwe}.
		\begin{lemma} Let $(X,d)$ be a complete metric space and $\varphi:X\to (-\infty,+\infty]$ lower semicontinuous, bounded from below and not identically equal to $+\infty$. Let $\e>0$ and $\lambda>0$ be given and let $x\in X$ be such that $\varphi(x)\leq \inf_X \varphi + \e$. Then, there exists $x_\e\in X$ such that
			\begin{enumerate}
				\item $\varphi(x_\e)\leq \varphi(x)$,
				\item $d(x_\e,x)\leq \lambda$, and
				\item $\varphi(x_\e)<\varphi(z)+\frac{\e}{\lambda}d(x_\e,z)$ for every $z\neq x_\e$.
			\end{enumerate}
		\end{lemma}
	
	\begin{proof}[Proof of Proposition \ref{Ekeland}]
		Let $p\in \mathscr{S}$. Take geodesic normal coordinates around $p$, valid in a small geodesic ball $B$, and $(y_i)\to 0$ a sequence in $B$ such that $u_i(y_i)\to +\infty$. We define
		\begin{equation}
			\e_i := u_i(y_i)^{-\frac{p_i-1}{2}}\to 0.
		\end{equation}
		We apply Ekeland's variational principle taking $\varphi = u_i^{-\frac{p_i-1}{2}}$ and $\lambda = \sqrt{\e_i}$. Then, there exists a sequence $(z_i)$ such that
		\begin{enumerate}
			\item $u_i(z_i)^{-\frac{p_i-1}{2}} \leq u_i(y_i)^{-\frac{p_i-1}{2}}$. Hence, $u_i(y_i)\leq u_i(z_i)$ and, consequently, $u_i(z_i)\to +\infty.$
			\item $\abs{z_i-y_i}\leq \sqrt{\e_i}$. In particular $(z_i)\to 0.$
			\item $u_i(z_i)^{-\frac{p_i-1}{2}}<u_i(y_i)^{-\frac{p_i-1}{2}}+\sqrt{\e_i}\abs{z_i-z}$ for every $z\neq z_i.$
		\end{enumerate}
		The idea is that the new sequence $(z_i)$ is more adequate to rescale and pass to the limit. Now, we set $\delta_i = u_i(z_i)^{-\frac{p_i-1}{2}}\to 0$, $B_i = B(z_i,\frac{r}{2})\cap B$ and define the sequence of rescaled functions
		\begin{equation}\label{tfgc7}
			v_i(x)= \delta_i^{\frac{2}{p_i-1}}u_i(\delta_i x + z_i),
		\end{equation}
		for every $x\in \tilde{B}_i = \frac{1}{\delta_i}B_i.$ Clearly
		\begin{equation} \label{tfgc3}
			v_i(0)=\frac{u_i(z_i)}{u_i(z_i)}=1.
		\end{equation}
		We claim that, for every given $\e>0$ and $R>0$, $v_i$ satisfies the following uniform boundedness property:
		\begin{equation}\label{tfgc4}
			v_i(x)\leq 1+\e,\: \forall x\in \tilde{B}_i \mbox{ with } \abs{x}<R.
		\end{equation}
		Indeed, from $(3)$, if $\abs{z-z_i}<R\delta_i$, then:
		\begin{equation}\label{tfgc5}
			u_i(z)<\left(\frac{1}{1-\sqrt{\e_i}R}\right)^{\frac{2}{p_i-1}}u_i(z_i).
		\end{equation}
		Taking $z=\delta_i x + z_i$ in \eqref{tfgc5} and using  \eqref{tfgc7} we obtain:
		\begin{equation}\label{tfgc6}
			v_i(x)<\frac{1}{(1-\sqrt{\e_i}R)^{\frac{2}{p_i-1}}},
		\end{equation}
		which proves \eqref{tfgc4}. This allows us to make a scaling argument in the spirit of \cite{GidasSpruck}. We distinguish between two cases:
		
		\
		
		\textbf{Case 1:} $p\in\partial M$ and, up to a subsequence, $\frac{d(z_i,\partial_0B)}{\delta_i}\to t_0 \geq 0$. 
		
		\
		
		Straightforward computations show that the function $v_i$ satisfies 
		\begin{align*}
			&-(\sqrt{\abs{g}}g^{jk})(\delta_ix+z_i)\frac{\partial^2v_i(x)}{\partial x_j\partial x_k}+\delta_i \frac{\partial(\sqrt{\abs{g}}g^{jk})}{\partial x_j}(\delta_ix+z_i)\frac{\partial v_i(x)}{\partial x_k}+\nonumber\\&+\delta_i^2 \frac{1}{C_n}(\sqrt{\abs{g}}R_i)(\delta_ix+z_i)v_i(x)-\frac{1}{C_n}(\sqrt{\abs{g}}K_i)(\delta_ix+z_i)v_i(x)^{p_i}=0
		\end{align*}
		for $x\in \tilde{B}_i$, and
		\begin{align*}
			g_{jj}\frac{\partial v_i(x)}{\partial x_j} \eta^j-\frac{n-2}{2}H_i(\delta_ix+z_i)v_i(x)^{\frac{p_i+1}{2}}=0
		\end{align*}
		on the straight portion of its boundary. Using \eqref{tfgc4} and local regularity estimates 
		we obtain that, up to a subsequence, $v_i\to v$ in $C^2_{loc}(\R^n_+)$. Also, $g_{jk} = \delta_{jk}+o(\abs{x}^2)$ so, if we let $i\to +\infty$, we get a solution to the limit problem \eqref{lp1}.
		
		The results in \cite{ChipotFilaShafrir} (recalled in Proposition \ref{lp2} above) establish that this problem admits a solution whenever $\D_n(p)\geq 1.$
		
		\
		
		\textbf{Case 2:} $\frac{d(x_k,\partial_0B)}{\delta_k}\to +\infty$.
		
		\
		
		In this situation, the domains $\tilde{B}_i$ invade all of $\R^n$: reasoning as before, we have $v_i\to v$ in $C^2_{loc}(\R^n)$. Moreover, $v$ solves the following equation:
		\begin{equation}
			-\Delta v = \frac{n-2}{4(n-1)}K(0)v^{\frac{n+2}{n-2}} \mbox{ on } \R^n.
		\end{equation}
		Consider $R>0$ and the domain $\Omega = B^n(0,R)$. Since $K<0$, we have $\Delta v >0$ in $\Omega$ so, by the Maximum Principle, it cannot achieve its maximum unless it is constant. However, passing \eqref{tfgc6} to the limit we obtain
		\begin{equation*}
			v(x)\leq 1 \hsp \forall x \in \Omega,
		\end{equation*}
		while $v(0)=1$. Thus, this second case can be dismissed and the proof of the proposition is completed.
		
	\end{proof}
	
	As remarked previously, we cannot guarantee that around singular points we have a sequence of local maxima. We now show that, a posteriori, this is the case for singular points in $\{\D_n>1\}$.
	
	\begin{lemma} \label{l:maxloc} Let $p \in \cS_1$. Then there exists a sequence $x_i \in \partial M$, $x_i \to p$, such that $u_i(x_i) \to +\infty $ is local maximum of $u_i$.  Moreover, there exist $\oR_i\to+\infty$ and such that, up to a subsequence,
		\begin{align*}
		&r_i:= \oR_iu_i(x_i)^{-\frac{p_i-1}{2}} \to 0 \ \mbox{ and} \\
		&\left\Vert u_i(x_i)^{-1}u_i\left(u_i(x_i)^{-\frac{p_i-1}{2}}x+x_i\right)-b_{\beta_0}(x) \right\Vert_{C^2\left(B\left(x_i,\oR_i\right) \right)} \to 0,
		\end{align*}
		where $b_{\beta_0}$ is given by \eqref{lp4}.
	\end{lemma}
	
	\begin{proof} Given $p \in \cS_1$, we come back to the proof of Proposition \ref{Ekeland} (case 1). Hence we obtain that the rescaled functions $v_i$ converge locally in the $C^1$ sense to the function in \eqref{lp4}. Observe now that the limit solution \eqref{lp4} has a global maximum. As a consequence, the sequence $v_i$ attains also its maximum at a certain point $\tilde{x}_i$. Rescaling back, we obtain points $x_i$ which are local maxima of $u_i$.

		Let us now show that if $i$ is large enough, $x_i\in\partial M$. This is a consequence of the subharmonicity, given in turn by the hypothesis $K<0$, since
		\begin{equation*}
			-\Delta u_i(x_i) = K(x_i)u_i(x_i)^{p_i}-S \, u_i(x_i)<0.
		\end{equation*}
	The convergence is due to the convergence of the newly defined sequence of rescaled functions $v_i$, as in the proof of Case 1 of Proposition \ref{Ekeland}
	\end{proof}

\section{Properties of isolated simple blow-up points} \label{s:i-s-bu} \setcounter{equation}{0}

Along this section we analyze blow-up points $x \in \cS_1$ and the sequence of maxima $(x_i)\to x$ given by Lemma \ref{l:maxloc}, where locally the blow-up has profile given by \eqref{lp4}. 
Below, we define the notions of isolated and simple blow-up points, first introduced in  \cite{schoen}, adapted to our framework. Then we study the asymptotic behavior of $u_i$ around such points. This analysis is the counterpart 
of some results from \cite{li95} and \cite{djadlimalchiodiahmedou} (see also \cite{felli}), which we adapt here 
to the case of negative scalar curvature. From this study, one obtains that there are only finitely many isolated simple blow-up points. In next section we shall prove that if $n=3$ all blow-up points of $\mathscr{S}_1$ are isolated and simple. In particular this implies (2.1) of Theorem \ref{compactness}.

In the asymptotic analysis that follow, we sometimes use the notation of the Euclidean space, which makes no difference in the computations.

\begin{definition}\label{tec01} $\bar{x}\in\partial M$ is an isolated point of blow-up for $(u_i)$ if there exist  local maxima $(x_i)$ such that $(x_i)\to x$, $u_i(x_i)\to +\infty$ and 
	\begin{equation}\label{tec02}
	u_i(y)\leq \frac{C}{\dist(y,x_i)^\frac{2}{p_i-1}}, \hsp\mbox{ for every } y\in B(x_i,R)_+,
	\end{equation}
	where $C$ and $R$ are positive constants independent of $i$.
\end{definition}

\begin{definition}\label{tec03} Let $(u_i)$, $(x_i)$, $\bar{x}$ be as in Definition \ref{tec01}, and consider the radial averages
	\begin{equation*}
	\overline{u}_i(r) = \fint_{\de^+B(x_i,r)_+}u_i.
	\end{equation*}
	 $\bar{x}$ is called an isolated simple blow-up point of blow-up if there exists $\rho>0$, independent of $i$, such that the 1-dimensional functions
	\begin{equation*}
	\hat{u}_i(r)=r^{\frac{2}{p_i-1}}\overline{u}_i(r)
	\end{equation*}
	have exactly one critical point for $r\in (0,\rho)$.
\end{definition}

Next proposition gives a version of the well-known Harnack inequality: 
\begin{proposition}\label{tec05}
	Let $\Omega=B(p,R_0) \subset M$. Assume that $(u_i)$ is a sequence of solutions of \eqref{tfgc1} with $1\leq  p_i$, $p$ is an isolated point of blow-up. Let $(x_i) \to p$ and $R<R_0$ be as in Definition \ref{tec01}. Then, for every $0<r<\frac{R}{4}$, the following Harnack-type inequality holds
	\begin{equation*}
	\max_{B(x_i,2r)_+\backslash B(x_i,\frac{r}{2})_+}u_i(y)\leq C \min_{B(x_i,2r)_+\backslash B(x_i,\frac{r}{2})_+}u_i(y),
	\end{equation*}
	where $C$ is a positive constant independent of $i$ and $r$.
\end{proposition}

The details can be found in \cite{HanLi}. 


The main result of this section is the following description of the sequence $u_i$ around isolated simple blow-up points:

\begin{proposition}\label{tec2} Let us assume $\Omega=B_2=B(0,2)$ for simplicity. Assume $(K_i)\to K<0$, in $C^1(\overline{\Omega_+})$ and $(H_i)\to H$ in $C^2(\overline{\partial_0\Omega})$. Suppose that, for every $i\in\N$, $u_i$ is a positive solution of \eqref{tfgc1} and that $x_i\to 0$ is an isolated simple blow-up point with
\begin{equation}\label{tec20} u_i(x)\leq C_1 \abs{x-x_i}^{\frac{2}{p_i-1}}, \hsp \mbox{for all } x\in\Omega_+.
\end{equation}
Then there exists a constant $C=C(K,S,H,n,C_1)>0$ such that
\begin{equation}\label{tec4} u_i(x_i)u_i(x)\leq C\abs{x-x_i}^{2-n}, \hsp \mbox{for all } x\in B(x_i,1)_+.
\end{equation}
Moreover, there exist $a > 0$ and $b:\overline{(B_1)_+}\to\R$ satisfying 
\begin{equation} \label{eq:b}
\left\lbrace\begin{array}{ll}\Delta b=0 & \mbox{in } (B_1)_+ \\ \frac{\partial b}{\partial x_n} = 0 & \mbox{on } \partial_0(B_1)_+\end{array}\right.
\end{equation}
such that $u_i(x_i)u_i(x)\to a\abs{x}^{2-n}+b$ in $C^2_{loc}\left(\overline{(B_1)_+}\backslash\{0\}\right)$. Furthermore, 
for $\rho < 1$, in $\overline{(B_\rho)_+}$ one has that 
\begin{equation}\label{eq:1+o1}
	u_i = (1 + o_i(1) + o_\rho(1)) b_{\beta_i}; \qquad \nabla u_i = (1 + o_i(1) + o_\rho(1)) \nabla b_{\beta_i}, 
\end{equation}
with $\beta_i \to + \infty$, where $b_\beta$ is as in \eqref{lp4}. 
\end{proposition}
We split the proof of Proposition \ref{tec2} into some lemmas. As a preliminary result, we outline a version of the Maximum Principle  that can be found in \cite{HanLi}, pp. 539. 
\begin{lemma}\label{tec3} Let $\Omega$ be a bounded domain with piecewise smooth boundary $\partial\Omega = \partial_1\Omega\cup \partial_2\Omega$, $V\in L^\infty(\Omega)$ and $w\in L^\infty(\Omega)$. Suppose that $u\in C^2(\Omega)\cap C^1(\overline{\Omega})$, $u>0$, satisfies:
\begin{equation*}
\left\lbrace\begin{array}{ll} \Delta u + Vu \leq 0 & \mbox{in}\hsp \Omega \\ \frac{\partial u}{\partial \eta}-wu\geq 0 & \mbox{on}\hsp \partial_1\Omega\end{array}\right.
\end{equation*}
Then, for every $v\in C^2(\Omega)\cap C^1(\overline{\Omega})$ such that 
\begin{equation*}
\left\lbrace\begin{array}{ll} \Delta v + Vv \leq 0 & \mbox{in}\hsp \Omega \\ \frac{\partial v}{\partial \eta}-wv\geq 0 & \mbox{on}\hsp \partial_1\Omega \\ v\geq 0 & \mbox{on}\hsp \partial_2\Omega\end{array}\right.
\end{equation*}
we have $v\geq 0$ in $\overline{\Omega}$.
\end{lemma}
\begin{proof}
Apply the standard Maximum Principle to the function $w = \frac{v}{u}$.
\end{proof}

In next lemma we give a less sharp version of \eqref{tec4}.

\begin{lemma}\label{tec5} Under the same hypotheses of Proposition \ref{tec2}, there exists $\e_i>0,$ $\e_i = O\left({\oR_i}^{-2}\right)$, such that 
\begin{equation*}
u_i(x_i)^{\lambda_i}u_i(x)\leq C\abs{x-x_i}^{2-n+\e_i},\hsp \mbox{for all } r_i\leq \abs{x-x_i}\leq 1,
\end{equation*}
with $\lambda_i = (n-2-\e_i)\frac{p_i-1}{2}-1$.
\end{lemma}
\begin{proof} We  follow some ideas in \cite{li95}, together with the application of Lemma \ref{tec3}, as it is done in \cite{HanLi}.
Recall first the definition of $r_i$ in Lemma \ref{l:maxloc}. 
	
\

To begin, we claim that, for $\abs{x-x_i}=r_i$, the following bound holds
\begin{equation}\label{tec6} u_i(x)\leq C u_i(x_i){\oR_i}^{2-n}. 
\end{equation}
Indeed, rescaling back the second statement of Lemma \ref{l:maxloc}, we have
\begin{align*}
u_i(x)&\leq \frac{1}{u_i(x_i)}\frac{C}{\abs{x-x_i}^{n-2}} = \frac{1}{u_i(x_i)}\frac{C}{{r_i}^{n-2}} \\&= C{\oR_i}^{2-n}u_i(x_i)^{(n-2)\left(\frac{p_i-1}{2}\right)-1} \leq C{\oR_i}^{2-n}u_i(x_i).
\end{align*}
Consider the function $\tilde{u}_i:r\mapsto r^{\frac{1}{p_i-1}}\overline{u_i}(r)$. By our assumptions, it has a unique critical point in $(0,1)$. Moreover, by Lemma \ref{l:maxloc} we can always assume that $i$ is so large that the only critical point stays in $(0,r_i)$). Therefore, we can suppose without loss of generality that $\tilde{u}_i$ is strictly decreasing in $r_i<r<1$. Then,
\begin{align*}
\abs{x-x_i}^{\frac{2}{p_i-1}}u_i(x)\leq Cr_i^{\frac{2}{p_i-1}}\overline{u}_i(r_i)=C{\oR_i}^{\frac{2}{p_i-1}}u_i(x_i)^{-1}\fint_{\partial^+B(x_i,r_i)_+}u_i. 
\end{align*}
Using \eqref{tec6}, we conclude that 
\begin{equation*}
\abs{x-x_i}^{\frac{2}{p_i-1}}u_i(x)\leq C{\oR_i}^{-\frac{2}{p_i-1}}.
\end{equation*}
Consequently,
\begin{equation}\label{tec7}u_i(x)^{p_i-1}\leq C{\oR_i}^{-2+o_i(1)}\frac{1}{\abs{x-x_i}^2},\hsp \mbox{for all } r_i\leq \abs{x-x_i}\leq 1.
\end{equation}
Next, we apply a comparison argument. Consider the second-order elliptic operator
\begin{equation*}
\mathfrak{L}_i(\varphi)=\Delta\varphi + \frac{n-2}{4(n-1)}\left(K_iu_i^{p_i-1}+S\right)\varphi,
\end{equation*}
and the boundary operator 
\begin{equation*}
\mathfrak{B}_i(\varphi)=\frac{\partial \varphi}{\partial \eta}-\frac{n-2}{2}H_iu_i^{\frac{p_i-1}{2}}\varphi.
\end{equation*}
Since $u_i > 0$ solves \eqref{tfgc1}, $\mathfrak{L}_i(u_i)=\mathfrak{B}_i(u_i)=0$. Then, the pair $(\mathfrak{L}_i,\mathfrak{B}_i)$ satisfies the Maximum Principle stated in Lemma \ref{tec3}. We start constructing an adequate function to be compared with $u_i$. Considering $x\mapsto\abs{x-x_i}^{-\mu}$, it is clear that
\begin{equation*}
\Delta \abs{x-x_i}^{-\mu} = -\mu(n-2-\mu)\abs{x-x_i}^{-\mu-2}.
\end{equation*}
Then
\begin{align*}
\mathfrak{L}_i(\abs{x-x_i}^{-\mu})&=-\mu(n-2-\mu)\frac{1}{\abs{x-x_i}^{\mu+2}}\\&+C_n^{-1}K_iu_i^{p_i-1}\frac{1}{\abs{x-x_i}^\mu}-C_n^{-1}S\frac{1}{\abs{x-x_i}^{\mu}}.
\end{align*}
Using \eqref{tec7}, 
\begin{equation*}
\mathfrak{L}_i(\abs{x-x_i}^{-\mu}) = \left(-\mu(n-2-\mu)+C_n^{-1}K_i{\oR_i}^{-2}\right)\frac{1}{\abs{x-x_i}^{\mu+2}}-C_n^{-1}R_i\frac{1}{\abs{x-x_i}^{\mu}}.
\end{equation*}
Then, we can choose $\e_i \searrow 0,$ $\e_i = O\left({\oR_i}^{-2}\right)$ such that
\begin{equation}
\mathfrak{L}_i(\abs{x-x_i}^{-\e_i})\leq 0, \hsp \mbox{and}\hsp \mathfrak{L}_i(\abs{x-x_i}^{2-n+\e_i})\leq 0 \label{tec8}.
\end{equation}
Moreover, notice that
\begin{align*}
\mathfrak{B}_i(\abs{x-x_i}^{-\mu}) &= \mu \frac{x^n}{\abs{x-x_i}^{\mu+2}}-\frac{n-2}{2}H_iu_i^{\frac{p_i-1}{2}}\frac{1}{\abs{x-x_i}^{\mu}} \\&= \mu \frac{x^n}{\abs{x-x_i}^{\mu+2}} + O\left({\oR_i}^{-2}\right)\frac{1}{\abs{x-x_i}^{\mu}},
\end{align*}
again by \eqref{tec7}. Reasoning as above, we also have
\begin{equation}
\mathfrak{B}_i(\abs{x-x_i}^{-\e_i})\geq 0, \hsp \mbox{and}\hsp \mathfrak{B}_i(\abs{x-x_i}^{2-n+\e_i})\geq 0. \label{tec9}
\end{equation}
Set $M_i = \max_{\partial^+(B_1)_+}u_i$, $\lambda_i = (n-2+\e_i)\frac{p_i-1}{2}-1$ and define 
\begin{equation*}
\varphi_i = M_i\abs{x-x_i}^{-\e_i}+\alpha u_i(x_i)^{-\lambda_i}\abs{x-x_i}^{2-n+\e_i}, \hsp \mbox{for all } r_i \leq \abs{x-x_i}\leq 1,
\end{equation*}
with $\alpha>0$ yet to be determined. In order to apply the Maximum Principle and compare $\varphi_i$ with $u_i$, we choose as domain the semi-annulus $A=B(x_i,1)_+\backslash B(x_i,r_i)_+$. Observe that $\partial^+A$ is composed by two semi-spheres, namely, 
\begin{equation*}
\partial^+ A = \{x\in A:\abs{x-x_i}=1\}\sqcup\{x\in A:\abs{x-x_i}=r_i\}.
\end{equation*}
By \eqref{tec8} and \eqref{tec9}, $\mathfrak{L}_i(\varphi_i-u_i)\leq 0$ in $A$ and $\mathfrak{B}_i(\varphi_i-u_i)\geq 0$ on $\partial_0 A$. Then, we just need to prove that $\varphi_i \geq u_i$ on the circular boundaries.

\

If $\abs{x-x_i}=1,$ then $\varphi_i = M_i +\alpha u_i(x_i)^{-\lambda_i}\geq u_i$ by the choice of $M_i$.

\

On the other hand, if $\abs{x-x_i}=r_i$, then
\begin{equation*}
\varphi_i = M_i{\oR_i}^{-\e_i}u_i(x_i)^{\e_i\frac{p_i-1}{2}}+\alpha u_i(x_i)^{1+o_i(1)}{\oR_i}^{2-n+o_i(1)}. 
\end{equation*}
By \eqref{tec6}, it is possible to choose $\alpha>0$ large enough so that $\varphi_i\geq u_i$ for $\abs{x-x_i}=r_i.$ Then, the application of Lemma \ref{tec3} to $\varphi_i-u_i$ gives:
\begin{equation*}
u_i(x)\leq M_i\abs{x-x_i}^{-\e_i}+\alpha u_i(x_i)^{-\lambda_i}\abs{x-x_i}^{2-n+\e_i}.
\end{equation*}
Finally, using the monotonicity of the weighed radial averages, for every $r_i<\theta<1$ we obtain 
\begin{align}\label{tec10}
M_i \leq C\overline{u_i}(1)\leq C\theta^{\frac{p_i-1}{2}}\overline{u_i}(\theta)\leq C\theta^{\frac{p_i-1}{2}}\left(M_i\theta^{-\e_i}+\alpha u_i(x_i)^{-\lambda_i}\theta^{2-n+\e_i}\right).
\end{align}
We can choose $\theta$ small enough (up to taking a larger $i$) and derive that 
\begin{equation*}
M_i\leq Cu_i(x_i)^{-\lambda_i}.
\end{equation*}
The proof can be concluded applying the Harnack inequality and absorbing the first term of the right-hand side of \eqref{tec10}.
\end{proof}
\begin{lemma}\label{tec11}Under the assumptions of Proposition \ref{tec2},
\begin{equation*}
\tau_i := \frac{n+2}{n-2} - p_i = O\left(u_i(x_i)^{\frac{-2}{n-2}+o_i(1)}\right),\hsp\mbox{as}\hsp i\to +\infty.
\end{equation*}
Therefore,
\begin{equation*}
u_i(x_i)^{\tau_i} = 1+o_i(1).
\end{equation*}
\end{lemma}
\begin{proof}
We apply Corollary \ref{pi4} for $u(x)=u_i(x+x_i)$ on $\Omega = B(0,1)$, and estimate all the terms using Lemmas \ref{l:maxloc} and \ref{tec5}:
\begin{align*}
\int_{\partial^+\Omega_+}\abs{K_i}{u_i}^{p_i+1}&\leq C\int_{\partial^+\Omega_+}\abs{K_i}u_i(x_i)^{\lambda_i(p_i+1)}\abs{x}^{(2-n+\e_i)(p_i+1)} \\ &= O\left( u_i(x_i)^{-(p_i+1)+o_i(1)}\right).
\end{align*}
Analogously,
\begin{align*}
\int_{\partial^+\Omega_+}\abs{S}{u_i}^2 &= O\left(u_i(x_i)^{-2+o_i(1)}\right), \\ \int_{\partial(\partial_0\Omega_+)}\abs{H_i}{u_i}^{\frac{p_i+3}{2}}&=O\left(u_i(x_i)^{-\frac{p_i+3}{2}+o_i(1)}\right).
\end{align*}
In order to simplify the notation, let $\delta_i=u_i(x_i)^{-\frac{p_i-1}{2}}$. We now bound the gradient terms using a rescaling argument:
\begin{align*}
&\int_{\Omega_+}u_i(x+x_i)^{p_i+1}\abs{X\cdot \nabla K_i(x+x_i)}dx \\&= u_i(x_i)^{-\frac{2}{n-2}+o_i(1)}\int_{\left(B_{\delta_i^{-1}}\right)_+}v_i(y)^{p_i+1}Y\cdot\nabla K_i(\delta_ix+x_i)dy,
\end{align*}
where $v_i(y)=u_i(\delta_iy+x_i)$. By Lemma \ref{l:maxloc}, $v_i\to b_\beta$ in $C^2_{loc}\left((B_{\delta_i^{-1}})_+\right)$ for some $\beta>0$. This, together with the asymptotic behaviour of $b_\beta$ described in \eqref{lp18} and the fact that $K_i(\delta_i x + x_i)\to K(0)$ uniformly, gives:
\begin{equation*}
\int_{\Omega_+}u_i(x+x_i)^{p_i+1}\abs{X\cdot \nabla K_i(x+x_i)}dx = O\left(u_i(x_i)^{-\frac{2}{n-2}+o_i(1)}\right). 
\end{equation*}
In the same way, we obtain:
\begin{align*}
\int_{\Omega_+}u_i(x+x_i)^2 \abs{X\cdot \nabla S(x+x_i)}dx &= O\left(u_i(x_i)^{-\frac{6}{n-2}+o_i(1)}\right), \\ \int_{\partial_0\Omega_+}u_i(x+x_i)^{\frac{p_i+3}{2}}\abs{X\cdot \nabla H_i(x+x_i)}dx &=O\left(u_i(x_i)^{-\frac{2}{n-2}+o_i(1)}\right),
\end{align*}
and, using the same argument,
\begin{equation*}
\int_{\Omega_+}\abs{S(x+x_i)}u_i(x+x_i)^2dx = O\left(u_i(x_i)^{-\frac{4}{n-2}+o_i(1)}\right).
\end{equation*}
Moreover, by standard elliptic theory, Proposition \ref{pri11} and Lemma \ref{tec5}, we have:
\begin{equation*}
\int_{\partial^+\Omega_+}B_1(u_i,\nabla u_i) = O\left(u_i(x_i)^{-2+o_i(1)}\right).
\end{equation*}
Again, by a rescaling argument we see that
\begin{align}
&\tau_i\left(\int_{B(x_i,1)_+}K_iu_i^{p_i+1}+2(n-1)\int_{\partial_0B(x_i,1)_+}H_iu_i^{\frac{p_i+3}{2}}\right) \nonumber\\ &= \tau_i\left(\int_{\R^n_+}K(0){b_\beta}^\dst+2(n-1)\int_{\partial\R^n_+}H(0){b_\beta}^{\dsh}+o_i(1)\right).\label{tec12}
\end{align}
We need to check that the coefficient of $\tau_i$ is positive. With this in mind, we recall that $b_\beta$ solves the problem
\begin{equation*}
\left\lbrace\begin{array}{ll}-\frac{4(n-1)}{n-2}\Delta b_\beta = K(0){b_\beta}^{\frac{n+2}{n-2}} & \mbox{in } \R^n_+, \\ \frac{2}{n-2}\frac{\de b_\beta}{\de \eta}=H(0){b_\beta}^{\frac{n}{n-2}}& \mbox{on }\de\R^n_+.\end{array}\right.
\end{equation*}
Multiplying by $b_\beta$ and integrating by parts, we obtain
\begin{equation}\label{tec13}
\int_{\R^n_+}K(0){b_\beta}^\dst+2(n-1)\int_{\partial\R^n_+}H(0){b_\beta}^{\dsh} = \frac{4(n-1)}{n-2}\int_{\R^n_+}\abs{\nabla b_\beta}^2 > 0.
\end{equation}
The combination of \eqref{tec12}, \eqref{tec13} and Corollary \ref{pi4} finally gives
\begin{equation}\label{tec14}
\tau_i \left(\int_{\R^n_+}\abs{\nabla b_\beta}^2+o_i(1)\right)=O\left(u_i(x_i)^{-\frac{2}{n-2}+o_i(1)}\right).
\end{equation}
This concludes the proof. 
\end{proof}
\begin{lemma}\label{tec15}Under the assumptions of Proposition \ref{tec2}, there holds:
\begin{equation*}
u_i(x_i)\int_{\partial^+(B_1)_+}\frac{\de u_i}{\de\eta} \to C  < 0,
\end{equation*}
for a negative constant $C = C(n,\beta)$.
\end{lemma}
\begin{proof}
Integrating \eqref{tfgc1} by parts, we get the relation
\begin{align}
&u_i(x_i)\int_{\de^+(B_1)_+}\frac{\de u_i}{\de\eta} = u_i(x_i)\frac{n-2}{4(n-1)}\int_{(B_1)_+}Su_i\nonumber\\&+ u_i(x_i)\left(\frac{n-2}{4(n-1)}\int_{(B_1)_+}\abs{K_i}{u_i}^{p_i}-\frac{n-2}{2}\int_{\de_0(B_1)_+}H_iu_i^{\frac{p_i+1}{2}}\right).\label{tec16}
\end{align}
Using Lemma \ref{tec5}, we can bound the previous integrals for $r_i\leq \abs{x-x_i}\leq 1$, as in \cite{djadlimalchiodiahmedou}:
\begin{align*}
\int_{(B_1)_+\backslash (B_{r_i})_+}{u_i}^{p_i} &\leq C \int_{(B_1)_+\backslash (B_{r_i})_+}u_i(x_i)^{-\lambda_ip_i}\abs{x-x_i}^{p_i(2-n+\e_i)}dx \\ &\leq C u_i(x_i)^{-1+o_i(1)}{\oR_i}^{-2+o_i(1)} = o_i(1)u_i(x_i)^{-1+o_i(1)},
\end{align*}
and, analogously
\begin{align*}
\int_{(B_1)_+\backslash (B_{r_i})_+}u_i &= O\left({\oR_i}^{2+o_i(1)}\right)u_i(x_i)^{-1+o_i(1)} \\ \int_{\de_0(B_1)_+\backslash \de_0(B_{r_i})_+}u_i^{\frac{p_i+1}{2}} &= o_i(1)u_i(x_i)^{-1+o_i(1)}.
\end{align*}
Then, \eqref{tec16} becomes:
\begin{align*}
&u_i(x_i)\int_{\de^+(B_{r_i})_+}\frac{\de u_i}{\de\eta} = u_i(x_i)\frac{n-2}{4(n-1)}\int_{(B_{r_i})_+}Su_i\nonumber\\&+ u_i(x_i)\left(\frac{n-2}{4(n-1)}\int_{(B_{r_i})_+}\abs{K_i}{u_i}^{p_i}-\frac{n-2}{2}\int_{\de_0(B_{r_i})_+}H_iu_i^{\frac{p_i+1}{2}}\right) \\ &+u_i(x_i)^{\tau_i}\left(o_i(1)+u_i(x_i)^{-\frac{4}{n-2}}{\oR_i}^{2+o_i(1)}\right).
\end{align*}
By Lemmas \ref{l:maxloc} and \ref{tec11} and a rescaling argument, we obtain:
\begin{align*}
u_i(x_i)\int_{\partial^+(B_1)_+}\frac{\de u_i}{\de\eta}&=\lim_{R\to+\infty}\int_{\de^+(B_R)_+}\frac{\de b_\beta}{\de\eta}+o_i(1) \\&=-\tilde{C}(\beta)\omega_{n-1}\frac{n-2}{2}+o_i(1) < 0,
\end{align*}
with $\tilde{C}(\beta)>0$.
\end{proof}
\begin{proof}[Proof of Proposition \ref{tec2}]
Inequality \eqref{tec4} for $0\leq \abs{x-x_i}\leq r_i$ is a consequence of Lemmas \ref{l:maxloc}, \ref{tec5} and \ref{tec11}, so let us address the inequality only for $r_i<\abs{x-x_i}\leq 1.$

\

Let $\overline{u_i}$ be the radial average of $u_i$, and consider $\overline{u_i}(1)$. By \eqref{tec7},
\begin{equation*}
u_i(x)=O\left({\oR_i}^{-\frac{2}{p_i-1}+o_i(1)}\right)\abs{x-x_i}^{-\frac{2}{p_i-1}}, \hsp\mbox{for}\hsp r_i\leq\abs{x-x_i}\leq 1,
\end{equation*}
then $\overline{u_i}(1)\leq C{\oR_i}^{-\frac{2}{p_i-1}}\to 0$, as $i\to+\infty.$ Define the sequence $\xi_i=\overline{u_i}(1)^{-1}u_i$. It is easy to see that $\xi_i$ satisfies
\begin{align*}
\left\lbrace\begin{array}{ll}-\Delta\xi_i = \frac{n-2}{4(n-1)}\overline{u_i}(1)^{p_i-1}K_i{\xi_i}^{p_i}-\frac{n-2}{4(n-1)}S\xi_i & \mbox{in}\hsp (B_2)_+\backslash\{0\}, \\ \frac{\de\xi_i}{\de\eta} = \frac{n-2}{2}\overline{u_i}(1)^{\frac{p_i-1}{2}}H_i{\xi_i}^{\frac{p_i+1}{2}}&\mbox{on}\hsp\partial_0(B_2)_+\backslash\{0\}.\end{array}\right.
\end{align*}
Harnack's Inequality holds in the annulus $(B_2)_+\backslash\{0\}$ and we can pass to the limit to find a function $h$ such that $\xi_i\to h$ in $C^2_{loc}\left((B_2)_+\backslash\{0\}\right)$. This function verifies
\begin{align}\label{tec17}
\left\lbrace\begin{array}{ll}-\Delta h = 0 & \mbox{in}\hsp (B_2)_+\backslash\{0\}, \\ \frac{\de h}{\de\eta} = 0&\mbox{on}\hsp\partial_0(B_2)_+\backslash\{0\}.\end{array}\right.
\end{align}
Since the origin is an isolated simple blow-up point, $h$ must be singular at $0$. Equation \eqref{tec17} allows us to consider the symmetric extension of $h$ to $B_2\backslash\{0\}$, given by
\begin{equation*}
\tilde{h}(x)=\left\lbrace\begin{array}{ll}h(x)&\mbox{if}\hsp x_n\geq0, \\ h(x^1,\ldots,x^{n-1},-x^n)&\mbox{otherwise},\end{array}\right.
\end{equation*} 
which is positive and harmonic in $B_2$. By uniqueness of the harmonic extension, it must coincide with the one given by  Schwartz's Reflection Principle, so we can write
\begin{equation*}
h(x)=a\abs{x}^{2-n}+b(x),
\end{equation*}
where $a$ is a positive constant and $b$ satisfies 
\begin{align*}
\left\lbrace\begin{array}{ll}-\Delta b = 0 & \mbox{in}\hsp (B_1)_+\backslash\{0\}, \\ \frac{\de b}{\de\eta} = 0&\mbox{on}\hsp\partial_0(B_1)_+\backslash\{0\}. \end{array}\right.
\end{align*}
Let us prove \eqref{tec2} for $\abs{x-x_i}=1$, namely, $\overline{u_i}(1)\leq Cu_i(x_i)^{-1}$. By the harmonicity of $b$,
\begin{equation*}
0 = \int_{(B_1)_+}\Delta b = \int_{\partial^+(B_1)_+}\frac{\partial b}{\de \eta},
\end{equation*}
from which we deduce that
\begin{equation}\label{tec18} \lim_{i\to+\infty}\overline{u_i}(1)^{-1}\int_{\de^+(B_1)_+}\frac{\de u_i}{\de \eta}=\int_{\de^+(B_1)_+}\frac{\de h}{\de \eta}=a\int_{\de^+(B_1)_+}\frac{\de \abs{x}^{2-n}}{\de \eta}<0.
\end{equation}
The result follows from Lemma \ref{tec15}. For the general case $r_i<\abs{x-x_i}<1$ we use a rescaling argument to reduce ourselves to the case $\abs{x-x_i}=1$, as in \cite{li95}. Assume by contradiction that there exists $\tilde{x}_i$, $r_i\leq \abs{\tilde{x}_i-x_i}\leq 1$, such that
\begin{equation}\label{tec19} \lim_{i\to+\infty}u_i(\tilde{x}_i)u_i(x_i)\abs{\tilde{x}_i-x_i}^{n-2}=+\infty.
\end{equation}
Set $\tilde{r}_i = \abs{\tilde{x}_i-x_i}$, and $\tilde{u}_i={\tilde{r}_i}^{\frac{2}{p_i-1}}u_i(\tilde{r}_ix+x_i).$ As in Section \ref{s:blowup}, one can prove that $\tilde{u}_i$ solves the boundary value problem
\begin{align*}
\left\lbrace\begin{array}{ll} \frac{4(n-1)}{n-2}\Delta \tilde{u}_i(x) = K_i(\tilde{r}_ix+x_i)\tilde{u}_i(x)^{p_i}-{\tilde{r}_i}^2S(\tilde{r}_ix+x_i)\tilde{u}_i(x)&\mbox{in}\hsp (B_{\tilde{r}_i^{-1}})_+, \\ \frac{2}{n-2}\frac{\de \tilde{u}_i(x)}{\de \eta}=H_i(\tilde{r}_ix+x_i)\tilde{u}_i(x)^{\frac{p_i+1}{2}}&\mbox{on}\hsp \partial_0(B_{\tilde{r_i}^{-1}})_+.\end{array} \right.
\end{align*}
We claim that $0$ is  isolated simple  for $\tilde{u}_i.$ First, observe that
\begin{equation*}
\tilde{u}_i(0)={\tilde{r}_i}^{\frac{2}{p_i-1}}u_i(x_i)\geq {r_i}^{\frac{2}{p_i-1}}u_i(x_i) = {\oR_i}^{\frac{2}{p_i-1}}\to +\infty.
\end{equation*}
Moreover, rescaling \eqref{tec20}, we get:
\begin{equation*}
\tilde{u}_i(x)={\tilde{r}_i}^{\frac{2}{p_i-1}}u_i(\tilde{r}_ix+x_i)\leq C_1\abs{x}^{-\frac{2}{p_i-1}},\hsp \mbox{for}\hsp \abs{x}\leq\frac{2}{\tilde{r}_i}.
\end{equation*}
Finally, it is easy to check that the weighed radial averages of $\tilde{u}_i$ and $u_i$ verify the relation
\begin{align}
\widehat{(\tilde{u}_i)}(r)=r^{\frac{2}{p_i-1}}\overline{\tilde{u}_i}(r)&=r^{\frac{2}{p_i-1}}\frac{2}{\omega_{n-1}r^{n-1}}\int_{\partial^+(B_r)_+}{\tilde{r}_i}^{\frac{2}{p_i-1}}u_i(\tilde{r}_ix+x_i)dx\nonumber\\&=(r\tilde{r}_i)^{\frac{2}{p_i-1}}\fint_{\partial^+\left(B(x_i,r\tilde{r}_i)\right)_+}u_i(y)dy=\hat{u_i}(r\tilde{r}_i),\label{tec2000}
\end{align}
from which it follows that $\widehat{(\tilde{u}_i)}$ has a unique critical point in the interval $(0,\tilde{r}_i^{-1})$, concluding the proof of the claim. Therefore, the hypotheses of Proposition \ref{tec2} hold for $\tilde{u}_i$, and in particular \eqref{tec4} in the unit sphere. This gives
\begin{equation*}
\tilde{u}_i(0)\tilde{u}_i\left(\frac{\tilde{x}_i-x_i}{\tilde{r}_i}\right)={\tilde{r}_i}^{\frac{4}{p_i-1}}u_i(\tilde{x}_i)u_i(x_i)\leq C, \hsp\mbox{for all}\hsp i.
\end{equation*} 
Hence, 
\begin{equation*}
u_i(\tilde{x}_i)u_i(x_i)\abs{\tilde{x}_i-x_i}^{n-2} \leq C{\tilde{r}_i}^{\tau_i},
\end{equation*}
contradicting \eqref{tec19}.

Concerning the final statement of the proposition, \eqref{eq:1+o1} clearly holds in a 
ball of size $R_i u_i(x_i)^{-\frac{p_i-1}{2}}$, by Lemma \ref{l:maxloc}. Notice then that 
by \eqref{tec4} and Lemma \ref{tec11}, the measures $K_i u_i(x_i) u_i^{p_i}$ and 
$H_i u_i^{\frac{p_i+1}{2}}$ converge to Dirac masses at the origin on the closure 
of $(B_1)_+$ and on $\partial_0 (B_1)_+$ respectively. This implies that 
$u_i(x_i) u_i(x) \to a |x|^{2-n} + b$, with $a$ and $b$ as in the statement, 
also uniformly  on $(B_1(x_i))_+ \setminus (B_{R_i u_i(x_i)^{-\frac{p_i-1}{2}}})_+$, 
which proves the first estimate in \eqref{eq:1+o1}. The gradient estimate 
follows by standard regularity results. 
\end{proof}

\section{Blow-up points in $\mathscr{S}_1$ are isolated and simple for $n=3$}\setcounter{equation}{0} \label{s:n=3}

In this section we prove the following result:

\begin{theorem}\label{3d15} Suppose that $n=3$. Then $\mathscr{S}_1$ consists of isolated simple points of blow-up. In particular, in view of Proposition \ref{tec2}, $\mathscr{S}_1$ is finite.
\end{theorem}

In order to prove Theorem \ref{3d15}, some preliminary results are needed. The following proposition can be proved as in \cite{HanLi}, with minor modifications.
\begin{proposition}\label{3d1}Let $(u_i)$ be a blowing-up sequence of solutions of \eqref{tfgc1}. Given a large $\oR>0$ and a small $\e>0$, for large enough $i$ there exists a constant $C=C(\oR,\e)>0$ and a finite set of points $\left\lbrace q_1^i,\ldots,q_{m_i}^i\right\rbrace$, with $m_i\geq 1$, such that each $(q_j^i)$ is a local maximum for $u_i$ and satisfies:
\begin{enumerate}
	\item $\left\lbrace B\left(q_j^i,r_j^i\right)_+:\hsp j=1,\ldots,m_i\right\rbrace$ is a disjoint collection for $r_j^i= \oR u_i(q_j^i)^{-\frac{p_i-1}{2}}$,
	\item If $y = (y_1,\ldots,y_n)$ are geodesic coordinates centered at $q_j^i$, then
	\begin{equation}\label{3d10}
	\norm{u_i(q_j^i)^{-1}u_i\left(u_i(q_j^i)^{-\frac{p_i-1}{2}}y\right)-b_{\beta_j}(y)}_{C^2(B_{\oR})}<\e, 
	\end{equation}
	with $b_{\beta_j}$ being a solution of \eqref{lp1} in $\R^n_+$ of the form \eqref{lp4}.
	\item \begin{equation}\label{3d2}
	u_i(x)\leq \frac{C}{\dist(x,\left\lbrace q_1^i,\ldots,q_N^i\right\rbrace)^{\frac{2}{p_i-1}}} \hsp \mbox{ for every } x\in M.
	\end{equation}
	Moreover, $u_i(q_j^i)\dist(q_j^i,q_k^i)\geq C^{-1}$ for every $j\neq k$.
\end{enumerate}
\end{proposition}

We highlight what condition (3) locally means; if $y=(y_1,\ldots,y_n)$ is a normal coordinate system centered at $q_j^i$, then
\begin{equation}\label{3d3}
u_i(y) \leq \frac{C}{\dist(y,q_j^i)^{\frac{2}{p_i-1}}},\hsp\mbox{for}\hsp y\in B(q_j^i,r_j^i)_+. 
\end{equation}
\begin{remark} The essential difference between \eqref{3d3} and \eqref{tec02} is that, in the former, the radius depends on $i$ and could colapse.
\end{remark}
In the sequel, we follow the steps in \cite{khurimarquesshoen}, using the assumption $n=3$ as in \cite{djadlimalchiodiahmedou}. We start proving that every isolated blow-up point in $\mathscr{S}_1$ is, in fact, isolated simple.
\begin{proposition}\label{3d7} Let $n=3$. If $q\in \mathscr{S}_1$ is an isolated point of blow-up, then it is isolated simple.
\end{proposition}
\begin{proof}
We proceed by contradiction, recalling Definition \ref{tec03}. If $(q_i)\to q\in \mathscr{S}_1$ is not isolated simple, then there are at least two critical points of $\hat{u}_i$ in the interval $(0,\overline{t}_i)$, for some sequence $\overline{t}_i\to 0$. 

We know by Lemma \ref{l:maxloc} that the sequence $u_i$, after rescaling it as in that statement, 
converges to a bubble $b_{\beta}$ on arbitrarily large balls, so rescaling back we see that there is, at most, one critical point in the interval $\left(0,\oR_i u_i(q_i)^{-\frac{p_i-1}{2}}\right)$ for the counterpart of the 
function $\hat{u}_i$ in Definition \ref{tec03}. Therefore the second critical point, that we call $t_i$, must verify
\begin{equation}\label{3d11} 
\frac{\oR_i}{u_i(q_i)^\frac{p_i-1}{2}}\leq t_i \leq \overline{t}_i.
\end{equation}
Now, take  normal coordinates $y=(y_1,\ldots,y_n)$ centered at $q_i$, and rescale $u_i$ in the following way:
\begin{equation*}
w_i(y)=t_i^{\frac{2}{p_i-1}}u_i(t_iy),\hsp\mbox{for}\hsp y\in B\left(0,t_i^{-1}\right)_+.
\end{equation*}
We claim that $w_i$ has an isolated simple blow-up point at the origin 
according to Definition \ref{tec03}, after properly dilating the arguments of the  
functions $K_i$ and $H_i$. 

\

Since $q$ is an isolated blow-up point for $u_i$, \eqref{tec02} holds in a ball of fixed radius $\rho>0$. As this inequality is 
scale-invariant, the same stands for $w_i$ for $\abs{y}<\rho t_i^{-1}$. Moreover, inequality \eqref{3d11} implies that
\begin{equation*}
w_i(0)={t_i}^{\frac{2}{p_i-1}}u_i(0)\geq \oR_i \to +\infty,
\end{equation*}
so $\{0\}$ is an isolated blow-up point for $w_i$. The next step is to show that the weighted radial averages $\hat{w}_i$ have a unique critical point in the interval $(0,1)$. This follows from \eqref{tec2000} together with the fact that, by definition of $t_i$, $u_i$ has a unique critical point in $(0,t_i)$, proving the above claim. 

We highlight that this also implies, via the chain rule, that
\begin{equation}\label{3d12}
\left.\frac{d}{dr}\right|_{r=1}\hat{w}_i(r) = t_i \hat{u}_i'(t_i)=0.
\end{equation}
The sequence $w_i$ satisfies the hypotheses of Proposition \ref{tec2}, so after passing to a subsequence we have
\begin{equation*}
w_i(0)w_i(y)\to h(y)=a\abs{y}^{2-n}+b(y), \hsp \mbox{locally in } C^2(\R^n_+\backslash\{0\}), 
\end{equation*}
with $a > 0$ and $b$ as in \eqref{eq:b}. 
We consider the function $b^*:\R^n\to \R$ defined by symmetrization as
\begin{equation}\label{3d17}
b^*(x)=\left\lbrace\begin{array}{ll} b(x_1,\ldots,x_n) & \mbox{ if } x\in \R^n_+ \\ b(x_1,\ldots,x_{n-1},-x_n) & \mbox{ if } x\in \R^n_-
\end{array}\right.,
\end{equation}
and we notice that it is  harmonic, by the fact that $b$ satisfies Neumann boundary conditions. Also, since $h(y)>0$, $\liminf_{\abs{y}\to+\infty}b^*(y)\geq 0$, so $b^*$ (and consequently $b$) is constant by Liouville's theorem, and using \eqref{3d12} we obtain:
\begin{equation*}
0 = \left.\frac{d}{dr}\right|_{r=1}\frac{\hat{h}(r)}{w_i(0)} = \frac{1}{w_i(0)}\left.\frac{d}{dr}\right|_{r=1}\left(\frac{a}{r^{\frac{n-2}{2}}}+br^{\frac{n-2}{2}}\right)=a-b.
\end{equation*}
Then, 
\begin{equation*}
w_i(0)w_i(y)\to a\abs{y}^{2-n}+a.
\end{equation*}
Proposition \ref{pri11} then gives
\begin{equation}\label{3d13}
\lim_{r\to 0}\int_{\partial^+(B_r)_+}B(w_i,\nabla w_i)=-(n-1)(n-2)\omega_{n-1}\frac{a^2}{w_i(0)^2}.
\end{equation}
On the other hand, we can write $w_i$ in terms of the limiting profile. Set $\delta_i = u_i(0)^{-\frac{p_i-1}{2}}$, and recall that $v_i(y)=\delta_i^{\frac{2}{p_i-1}}u_i(\delta_iy)$ satisfies $v_i(y)\to b_{\beta}(y)$ in the $C^2$ sense on balls of radius $\oR_i$. With this in mind, we can define $\lambda_i:=t_i\delta_i^{-1}\geq \oR_i$ and write
\begin{equation*}
w_i(y) =\lambda_i^\frac{2}{p_i-1}\delta_i^\frac{2}{p_i-1}u_i(\lambda_i\delta_iy)\to \lambda_i^\frac{2}{p_i-1}b_{\beta}(\lambda_iy).
\end{equation*}
With this notation, \eqref{3d13} becomes 
\begin{equation}\label{3d14}
\lim_{r\to 0}\int_{\partial^+(B_r)_+}B(w_i,\nabla w_i)=-(n-1)(n-2)\omega_{n-1}\frac{a^2}{\lambda_i^{n-2+O(\tau_i)}}.
\end{equation}
Now, we apply Corollary \ref{pi4} with $u(y)=w_i(y)$, $f(y)=K_i(t_iy)$, $g(y)={t_i}^2S(t_iy)$ and $h(y)=h_i(t_iy)$. We estimate all the  terms involved using the asymptotic behaviour of $b_\beta$ described in \eqref{lp18}, Lemma \ref{tec11} and a rescaling argument:
\begin{align*}
\int_{(B_r)_+}w_i(x)^{p_i+1}\abs{X\cdot \nabla K_i(t_ix)}dx &= O\left(\lambda_i^{-1+O(\tau_i)}\right)o_i(1), \\
\int_{(B_r)_+}{t_i}^2w_i(x)^2\abs{X\cdot \nabla S(t_ix)}dx &= O\left({\lambda_i}^{-3+O(\tau_i)}\right)o_i(1), \\
\tau_i\int_{(B_r)_+}K_i(t_ix)w_i(x)^{p_i+1}dx &= O\left(\lambda_i^{-1+O(\tau_i)}\right)o_i(1),\hsp\mbox{and}\\
\int_{(B_r)_+}{t_i}^2S(t_ix)w_i(x)^2dx &=O\left({\lambda_i}^{-2+O(\tau_i)}\right)o_i(1).
\end{align*}
We remark  that, by Proposition \ref{tec11}
\begin{equation}\label{3d20}
\tau_i = O\left(u_i(0)^{-\frac{2}{n-2}}\right) = O\left({\delta_i}^{-1+o_i(1)}\right)=O\left({\lambda_i}^{-1+o_i(1)}\right)t_i.
\end{equation}
In a similar way, we estimate
\begin{align*}
\int_{\de^+(B_r)_+}{t_i}^2S(t_ix)w_i(x)^2dx &=O\left({\lambda_i}^{-2+O(\tau_i)}\right)o_i(1), \\
\int_{\de^+(B_r)_+}K_i(t_ix)w_i(x)^{p_i+1}dx &= O\left({\lambda_i}^{-3+O(\tau_i)}\right),\\
\tau_i\int_{\de_0(B_r)_+}H_i(t_ix)w_i^{\frac{p_i+3}{2}}(x)dx &= O\left({\lambda_i}^{-n+O(\tau_i)}\right)o_i(1), \\
\int_{\de(\de_0(B_r)_+)}H_i(t_ix)w_i(x)^{\frac{p_i+3}{2}}\abs{X\cdot \nu}dx &= O\left({\lambda_i}^{-n+O(\tau_i)}\right),\hsp \mbox{and}\\
\int_{\de_0(B_r)_+}w_i(x)^{\frac{p_i+3}{2}}\abs{X\cdot\nabla H_i(t_ix)}dx&=O\left({\lambda_i}^{1-n+O(\tau_i)}\right)o_i(1).
\end{align*}
Then, taking $r>0$ small enough, by Corollary \ref{pi4}, \eqref{3d14} and the previous estimates:
\begin{equation}\label{3d21}
O\left({\lambda_i}^{-1+O(\tau_i)}\right)o_i(1) +O\left({\lambda_i}^{-2+O(\tau_i)}\right)= -(n-1)(n-2)\omega_{n-1}a^2{\lambda_i}^{2-n+O(\tau_i)}.
\end{equation}
Note that, by Lemma \ref{tec11} and \eqref{3d20},
\begin{equation*}
\lambda_i^{O(\tau_i)} = (1+o_i(1)){t_i}^{O(\tau_i)} = 1+o_i(1).
\end{equation*}
Therefore, \eqref{3d21} leads to a contradiction when $n=3$.
\end{proof}
We now proceed to rule out bubble accumulations.

\begin{proof}[Proof of Theorem \ref{3d15}]
Our goal is to prove that there exists $C>0$, independent of $i$, such that $\dist(q_j^i,q_k^i)\geq C$ for every $j\neq k$ in $\{1,\ldots,m_i\}$. Assume by contradiction that this is not the case. Then,
\begin{equation*}
\lim_{i\to +\infty} \min_{j\neq k} \dist(q_j^i,q_k^i) = 0.
\end{equation*}
Since the blow-up points $\{q_1^i,\ldots,q_{m_i}^i\}$ as in Proposition \ref{3d1} are finitely-many for every $i$, without loss of generality we can assume that 
\begin{equation*}
\sigma_i :=\min_{j\neq k} \dist(q_j^i,q_k^i) = \dist(q_1^i,q_2^i).
\end{equation*}
A direct application of item (3) in Proposition \ref{3d1} gives  that
\begin{equation*}
u_i(q_j^i)\sigma_i\geq \frac{1}{C}, \hsp\mbox{for } j=1,2.
\end{equation*}
Hence, $u_i(q_k^i)\to +\infty$ as $i\to +\infty$. Now, we take geodesic normal coordinates around $q_1^i$ and rescale the functions $u_i$ in the following way:
\begin{equation*}
v_i(y)=\sigma_i^{\frac{2}{p_i-1}}u_i(\sigma_iy), \hsp \mbox{for } y\in B\left(0,{\sigma_i}^{-1}\right)_+.
\end{equation*}
Moreover, if $q_k^i\in B\left(0,\frac{1}{\sigma_i}\right)$ and if we set $y_k^i=\frac{q_k^i}{\sigma_i}$, then each $y_k^i$ is a local maximum of $v_i$ and $\dist(y_1^i,y_2^i)=\abs{y_2^i}=1$, so up to a subsequence we can assume that $y_2^i\to y_2$ with $\abs{y_2}=1.$ 

\

We claim that both $y_1=0$ and $y_2$ are isolated blow-up points  for $v_i$. In  first place, we check that $v_i(y_j)\to +\infty$ for $j=1,2$.

\

If $v_i(y_2)$ remains bounded but $v_i(0)\to +\infty$, then $0$ is an isolated simple blow-up point for $v_i$ by Proposition \ref{3d7}, while the sequence is bounded from above around $y_2$. Then, by Proposition \ref{tec2}, $v_i(y_2)\to 0$. However, by item (1) of Proposition \ref{3d1} and the fact that the radii must be collapsing, for $\oR>0$,
\begin{equation}
\sigma_i\geq \max\left\lbrace \frac{\oR}{u_i(q_1^i)^{\frac{p_i-1}{2}}},\frac{\oR}{u_i(q_2^i)^{\frac{p_i-1}{2}}} \right\rbrace.\label{3d16}
\end{equation}
Rescaling back the previous inequality we obtain $\min\left\lbrace v_i(0),v_i(y_2) \right\rbrace\geq \oR$, contradicting $v_i(y_2)\to 0$. On the other hand, if both $v_i(0)$ and $v_i(y_2)$ are bounded, we can apply Harnack's inequality and find a limiting function $v_i\to v$ in $C^2_{loc}(\R^n_+)$ such that
\begin{align*}
\left\lbrace \begin{array}{ll} -\Delta v = \frac{n-2}{4(n-1)}K(p)v^{\frac{n+2}{n-2}} & \mbox{on }\R^n_+. \\[4px] \frac{\partial v}{\partial \eta} = \frac{n-2}{2}H(p)v^{\frac{n}{n-2}} & \mbox{on }\partial\R^n_+. \\[4px] \nabla v(0)=\nabla(y_2)=0.\end{array}\right.
\end{align*}
However, the classification given by \cite{ChipotFilaShafrir}, recalled in Proposition \ref{lp2}, yields  $v=0$, which again contradicts \eqref{3d16}. Now, observe that, by (1): 
\begin{equation*}
u_i(x)\leq \frac{C}{\abs{x-q_j^i}}^{\frac{2}{p_i-1}},\hsp\mbox{ for } \abs{x}\leq \frac{\sigma_i}{2}\hsp\mbox{ and}\hsp j=1,2.
\end{equation*}
Then, $v_i$ satisfies
\begin{equation*}
v_i(y)=\sigma_i^{\frac{2}{p_i-1}}u_i(\sigma_iy)\leq \frac{C}{\abs{y-y_j^i}^{\frac{2}{n-2}}}, \hsp \mbox{for }\abs{x}\leq \frac{1}{2}\hsp\mbox{and}\hsp j=1,2,
\end{equation*}
and the claim is proved. Proposition \ref{3d7} then guarantees that $0$ and $y_2$ are isolated simple blow-up points of $v_i$ (after a proper dilation of the arguments of $K_i$ and $H_i$, as in the proof of Proposition \ref{3d7}) and we can apply Proposition \ref{tec2} to obtain 
\begin{equation*}
v_k(0)v_i(y)\to h(y)=a_1\abs{y}^{2-n}+a_2\abs{y-y_2}^{2-n}+b(y), \hsp\mbox{locally in } C^2(\R^n_+\backslash \mathcal{S}),
\end{equation*}
with $\mathcal{S}$ being the blow-up set for $v_i$, and $b$ a function satisfying
\begin{align*}
\left\lbrace\begin{array}{ll} \Delta b = 0 & \mbox{on } \R^n_+\backslash \{\mathcal{S}\backslash \{0,y_2\}\}, \\\frac{\partial b}{\partial \eta }=0 & \mbox{on } \partial  \R^n_+\backslash \{\mathcal{S} \backslash \{0,y_2\}\}.\end{array}\right.
\end{align*}
Define 
\begin{equation*}
f(y)=a_2\abs{y-y_2}^{2-n}+b(y).
\end{equation*} 
It is clear that, for small $r>0$, $f\in C^1(\overline{B_+(r)})$ and, by the maximum principle, $f(0)>0$. Thus, we are under the assumptions of Proposition \ref{pri11} and we can reason as in the last part of Proposition \ref{3d7}.
\end{proof}

\section{Conclusion of the proof of Theorem \ref{compactness}} \label{s:pf-thm} \setcounter{equation}{0}


In this section we conclude the proof of Theorem \ref{compactness}. Let us recall that Proposition \ref{perturbed} and Theorem \ref{compactness} imply Theorem \ref{minmax}.

\subsection{Proof of Theorem \ref{compactness}, (2.2)}

In this subsection we prove that if $n=3$ and the scalar curvature $S$ satisfies $S \leq 0$, then $\mathscr{S}_1 = \emptyset$. For this purpose, we present a capacity argument that exploits fundamental differences between sequences of solutions of \eqref{tfgc1} and sequences of bubbles. Since we proved in the previous section that solutions resemble bubbles at scales of order $1$ near points in $\mathscr{S}_1$, we are able to dismiss this kind of blow-up in three dimensions. We have first  two propositions, 
which holds true for all dimensions  $n \geq 3$. 

\begin{proposition}\label{dtb1} Consider a sequence of positive solutions $(u_i)$  of \eqref{tfgc1} with $K_i\to K<0$ in $C^1(\overline{M})$, $S \leq 0$, and $H_i\to H$ in $C^2(\partial M)$. Then, given a small $\delta >0$, for large enough values of $i$, there exists a positive constant $C=C(\delta)$ such that
	\begin{equation}
	\int_{\{u_i\leq1\}}\frac{\modgrad{u_i}}{{u_i}^{\frac{p_i+3}{2}}} + \int_{\{u_i>1\}}\frac{\modgrad{u_i}}{{u_i}^{p_i+1+\delta}} < C. 
	\end{equation}
\end{proposition}
\begin{proof}
Consider the continuous function
\begin{align*}
f(u)=\left\lbrace\begin{array}{ll}
u^{-\left(\frac{p_i+1}{2}\right)} & \mbox{if } \:0 < u\leq 1, \\
u^{-\left(p_i+\delta\right)} & \mbox{if } \:u> 1.
\end{array}\right.
\end{align*}
 
Multiplying \eqref{tfgc1} by $f(u_i)$ and  integrating by parts, we obtain:
\begin{align*}
&\frac{4(n-1)}{n-2}\int_M f'(u_i)\modgrad{u_i} = -\int_M Su_if(u_i)+\int_{\{u_i\leq1\}}K_iu_i^{\frac{2}{n-2}+o_i(1)}\\&+\int_{\{u\geq 1\}}K_iu_i^{-\delta+o_i(1)}-2(n-1)\int_{\{u_1\leq1\}\cap\de M}H_i \\&+2(n-1)\int_{\{u_i>1\}\cap\de M}H_iu_i^{-\frac{2}{n-2}-\delta+o_i(1)}.
\end{align*}
Notice that all  terms in the right-hand side are uniformly bounded, with the possible exception of the first one: however, that term is non-negative by our assumptions.  Hence, we get an upper bound for the left-hand side: 
\begin{align*}
&\frac{4(n-1)}{n-2}\left\lbrace\frac{p_i+1}{2}\int_{\{u\leq 1\}}\frac{\modgrad{u}}{u^{\frac{p_i+3}{2}}}+\left(p_i-\delta\right)\int_{\{u> 1\}}\frac{\modgrad{u}}{u^{p_i+1+\delta}} \right\rbrace < C, 
\end{align*}
concluding the proof.
\end{proof}
Next, we show that the above property is not satisfied for the \emph{bubbles}, i.e. the solutions to \eqref{lp1} when $\D_n(p)>1$, described in \eqref{lp4}. We recall that this one-parameter family of solutions can be written as follows:
\begin{equation}\label{dtb4}
b_\beta(x)=\frac{C_n\beta^{\frac{n-2}{2}}}{\left(\abs{x-x_0(\beta)}^2-\beta^2\right)^{\frac{n-2}{2}}} \hsp\forall x\in \R^n_+,
\end{equation}
with $C_n=(n(n-2))^\frac{n-2}{4}$ and  $x_0(\beta)=\left(0,\ldots,0,-\D_n(p)\beta\right)$. Straightforward computations show that
\begin{equation*}
\frac{\modgrad{b_\beta}}{{b_\beta}^\mu} = \tilde{C}_n \beta^{(2-\mu)\frac{n-2}{2}}\frac{\abs{x-x_0(\beta)}^2}{\left(\abs{x-x_0(\beta)}^2-\beta^2\right)^{n-\mu\frac{n-2}{2}}},
\end{equation*}
where $\tilde{C}_n = (n-2)^2{C_n}^{2-\mu}$ and $\mu>0$. Moreover, the domain $\{b_\beta\leq 1\}$ is the complement of a ball centered in $x_0(\beta)$ with radius  tending to zero.  Notice that $b_\beta(x)\leq 1$ whenever $\abs{x-x_0(\beta )}^2 > \beta^2 + \sqrt{n(n-2)}\beta =: {r_\beta}^2$, thus 
\begin{equation*}
\{b_\beta\leq 1\}=\R^n_+ \backslash B^n(x_0(\beta),r_\beta).
\end{equation*}
For small enough values of $\beta$, we can take $0< r_\beta < r < R$ so that $$A_\beta(r,R) := A(x_0(\beta),r,R)\cap \R^n_+\subset \{b_\beta\leq 1\}.$$ The aim is now to prove that
\begin{equation}
\label{dtb2}\lim_{\beta\to0}\int_{A_\beta(r,R)}\frac{\modgrad{b_\beta}}{{b_\beta}^\mu} = +\infty 
\end{equation}
for some $0<\mu<\dsh$, obtaining an opposite conclusion to that of Proposition \ref{dtb1}. For a fixed $\beta>0$, 
\begin{equation*}
\int_{A_\beta(r,R)} \frac{\modgrad{b_\beta}}{{b_\beta}^\mu} = \tilde{C}_n\int_r^R\int_{\texttt{SC}^{n-1}(x_0(\beta),s,s-\beta\D_n(p))}\frac{\beta^{(2-\mu)\frac{n-2}{2}}s^2}{(s^2-\beta^2)^{n-\mu\frac{n-2}{2}}}dxds,
\end{equation*}
where $\texttt{SC}^{n-1}(x_0,r,h)$ denotes the $(n-1)-$dimensional spherical cap centered at $x_0\in\R^n$, with radius $r>0$ and height $0\leq h <r$. We know that 
\begin{equation}
\abs{\texttt{SC}^{n-1}(x,r,h)}= w_n r^{n-1} J(r,h),
\end{equation}
for some uniformly bounded function $J$. Thus, for some dimensional constant $\hat{C}_n$,
\begin{equation*}
\int_{A_\beta(r,R)} \frac{\modgrad{b_\beta}}{{b_\beta}^\mu} \geq \hat{C}_n\beta^{(2-\mu)\frac{n-2}{2}} \int_r^R\frac{s^{n+1}}{(s^2-\beta^2)^{n-\mu \frac{n-2}{2}}}ds. 
\end{equation*}
Since  $r> r_\beta > \beta$, the integral is uniformly bounded when $\beta\to0$, so it is enough to take $\mu > 2$ to have \eqref{dtb2}. We then proved the following result. 

\begin{proposition}\label{dtb3} Let $b_\beta$ be the family of functions in \eqref{dtb4}. Then,
	\begin{equation*}
	\lim_{\beta\to0}\int_{A_\beta(r,R)}\frac{\modgrad{b_\beta}}{{b_\beta}^\mu} = +\infty,
	\end{equation*}
	for all $A_\beta(r,R) \subset \{b_\beta \leq 1\}$ and $\mu >2$.
\end{proposition}

\

To prove that $\mathscr{S}_1$ is empty, it is then sufficient to combine Proposition \ref{dtb1}, 
Proposition \ref{dtb3}, Proposition  \ref{3d7} and \eqref{eq:1+o1} in Lemma \ref{l:maxloc}.

\subsection{Proof of Theorem \ref{compactness}, (2.3)}

In this subsection we consider  a sequence of solutions $(u_i)$ to \eqref{tfgc1} 
and we assume that $\mathscr{S}_0 \neq \emptyset$. We recall that the limit profile is one-dimensional, as given in \eqref{lp3}. In particular, we have that $\rho_i :=\int_{\de M}H_i{u_i}^{\frac{p_i+3}{2}}\to +\infty$. 

By assumption, $I_i(u_i)$ is uniformly bounded:
\begin{equation}\label{im2}
4\int_M \modgrad{u_i}+\frac{1}{2}\int_M S {u_i}^2-\frac{1}{p_i+1}\int_M\abs{K_i}{u_i}^{p_i+1} -\frac{2}{p_i+3}\int_{\de M}H_i{u_i}^{\frac{p_i+3}{2}} = O(1). 
\end{equation}
A second relation between the integrals is given by the fact that $I'(u_i)[u_i]=0$, namely
\begin{equation}\label{im3}
8\int_M\modgrad{u_i}+\int_M S{u_i}^2+\int_M\abs{K_i}{u_i}^{p_i+1}-4\int_{\de M}H_i{u_i}^{\frac{p_i+3}{2}} = 0. 
\end{equation}
Using the previous two relations, we will try to estimate the integrals in terms of $\rho_i$. First notice that, by 
the positivity of the first and third terms in \eqref{im3} and H\"{o}lder's inequality,
\begin{equation*}
\int_M S u_i^2 = O\left({\rho_i}^{\frac{2}{p_i+1}}\right).
\end{equation*}
Therefore, by \eqref{im2} and \eqref{im3}
\begin{equation}\label{im4}
8\int_M \modgrad{u_i} = \frac{1}{3}\int_M\abs{K_i}{u_i}^{p_i+1} +o(\rho_i) = \int_{\de M}H_i{u_i}^{\frac{p_i+3}{2}} +o_i(\rho_i).
\end{equation}
Now, take an arbitrary $\varphi \in C^2(M)$, multiply \eqref{tfgc1} by $u_i\varphi$ and integrate by parts. Then, the following identity holds:
\begin{align}\label{im5}
\int_M\left(8\abs{\nabla u_i}^2+\abs{K_i}{u_i}^{p_i+1}\right)\varphi -2\int_{\de M} H_i{u_i}^{\frac{p_i+3}{2}}\varphi \nonumber \\ = -\int_M S{u_i}^2\varphi-8\int_M u_i\nabla u_i\cdot \nabla \varphi.
\end{align}
By H\"{o}lder inequality
\begin{equation}\label{im6}
\int_M \abs{S}{u_i}^2\varphi = O\left({\rho_i}^{\frac{2}{\dst}}\right), \quad \hsp \int_Mu_i\nabla u_i\nabla \varphi = O\left({\rho_i}^{\frac{n-1}{n}}\right).
\end{equation}
The combination of \eqref{im4}, \eqref{im5} and \eqref{im6} yields that there exists a positive measure $\sigma$ defined in $\bar{M}$ such that 
\begin{enumerate}
	\item[(i)] $4\rho_i^{-1}H_i{u_i}^{\frac{p_i+3}{2}}\weakto \sigma|_{\partial M}$, and
	\item[(ii)] $\rho_i^{-1}\left(8\abs{\nabla u_i}^2+\abs{K_i}{u_i}^{p_i+1}\right)\weakto \sigma$ 
	weakly in the sense of measures. 
\end{enumerate} 
Observe that $\supp \sigma \subset \mathscr{S} \subset \{p\in \de M: \D_n(p)\geq 1\}$ by Theorem \ref{compactness}.

In dimension $n=3$, each blow-up point in $\mathscr{S}_1$ is isolated and simple by Theorem \ref{3d15}, and around 
such points we can control $u_i$ with good precision by 
Proposition \ref{tec2}. We combine both results to conclude the following:
\begin{lemma}\label{tfgc11} If $n=3$, then $\supp \sigma \subset \mathscr{S}_0$.
\end{lemma}
\begin{proof} Supose by contradiction that there exists $p\in \mathscr{S}_1$ and take $\varphi\in C^2(M)$, with compact support contained in $B_\e(p)$, such that 
\begin{equation}\label{im8}
\int_{\partial M}\varphi\:d\sigma >0.
\end{equation}
Then, for large enough $i$, there exists a constant $C>0$ such that 
\begin{equation*}
{\rho_i}^{-1}\int_{\partial M} H_i {u_i}^{\frac{p_i+3}{2}}\varphi ds_g = {\rho_i}^{-1}\int_{B_\e\cap \de M} H_i {u_i}^{\frac{p_i+3}{2}}\varphi ds_g > C^{-1}.
\end{equation*}
However, taking normal coordinates at $p$ and using Proposition \ref{tec2} jointly with 
Lemma \ref{tec11}, we find 
\begin{align*}
{\rho_i}^{-1}\int_{\partial M\cap B_\e} H_i {u_i}^{\frac{p_i+3}{2}}\varphi ds_g &\simeq {\rho_i}^{-1} H(p) \int_{\R^{2}}\frac{1}{(1+\abs{\tilde{x}}^2)^2}d\tilde{x} \\[0.3cm] & \leq {\rho_i}^{-1}O(1)\to 0,
\end{align*}
where '$\simeq$' represents equality up to a factor $(1+O(\e^2))(1+o_i(1))$ and $\tilde{x}=(x_1,\ldots,x_{n-1},0)$. This is a contradiction to \eqref{im8}.
\end{proof}
We can extract more information from the finite-energy hypothesis on $(u_i)$. In fact, following the steps of the proof for \eqref{ti2}, we obtain the following identities for the limiting measure $\sigma$.
\begin{proposition}\label{im9} For $n = 3$, let $(u_i)$ be a sequence of solutions of \eqref{tfgc1} with $I_i(u_i)$ uniformly bounded and $\mathscr{S}_0 \neq \emptyset$. Then,
	\begin{enumerate}
		\item[(i)] $4{\rho_i}^{-1}H_i{u_i}^{\frac{p_i+3}{2}}\weakto \sigma$,
		\item[(ii)] ${\rho_i}^{-1}\abs{K_i}{u_i}^{p_i+1} \weakto \frac{3}{4}\sigma$, and
		\item[(iii)] $8{\rho_i}^{-1}\modgrad{u_i} \weakto \frac{1}{4}\sigma.$
	\end{enumerate}
\end{proposition}
\begin{proof} Take any test function $\phi \in C^2_c(M)$ and a vector field $N$ on $M$ with $\abs{N}\leq 1$ and such that $N=\eta$ on $\partial M$. Straightforward computations show that, calling $X=H_i\phi u^{\frac{p_i+3}{2}} N$,
\begin{align} \dive(X) = u^{\frac{p_i+3}{2}} \left(H_i\phi\dive N + \phi N\cdot \nabla H_i+H_i\nabla \phi\cdot N+\frac{p_i+3}{2}H_i\phi \frac{\nabla u}{u}\cdot N\right).
\end{align}
Thus, by the Divergence Theorem, there exists a constant $C$, which depends on $\norm{\phi}_{C^1}$, $\norm{N}_{C^1}$ and $\norm{H}_\infty$, such that
\begin{equation*}
\int_{\de M}H_i\phi u^{\frac{p_i+3}{2}} \leq C\int_M u^{\frac{p_i+3}{2}} + \frac{p_i+3}{2} \int_M H_i \phi \abs{\nabla u}u^{\frac{p_i+1}{2}}.
\end{equation*}
By Cauchy-Schwartz's inequality,
\begin{align}\label{tfgc10}
\int_{\de M}H_i\phi u^{\frac{p_i+3}{2}} \leq C\int_M u^{\frac{p_i+3}{2}} + \frac{p_i+3}{2} \left(\frac{\delta^2}{2}\int_M \modgrad{u}\phi+\frac{1}{2\delta^2}\int_M H_i^2u^{p_i+1}\phi\right),
\end{align}
for every $\delta >0$. Now, let us define the measures $\sigma_1, \sigma_2$ by
\begin{enumerate}
	\item[(i)] $8{\rho_i}^{-1}\modgrad{u_i}\weakto \sigma_1$, and
	\item[(ii)] ${\rho_i}^{-1}\abs{K}{u_i}^{p_i+1} \weakto \sigma_2$.
\end{enumerate}
Notice that by \eqref{im4} and the definition of $\rho_i$ these measures are finite. 
By Lemma \ref{tfgc11}, we also know that $\sigma_1+\sigma_2 = \sigma,$ and that $\supp \sigma_j \subset \{\D_n = 1\}$. Multiplying by ${\rho_i}^{-1}$ and taking limits, using the fact that $H^2 =\frac 16 \abs{K}$ on $\{\D_n = 1\}$,
we obtain 
\begin{align*}
	\frac{1}{4}\int_{\de M}\phi\,d\sigma &\leq \frac{\delta^2}{4}\int_{\de M} \phi\,d\sigma_1+\frac{1+o_\e(1)}{3 \delta^2}\int_{\de M} \phi\,d\sigma_2.
\end{align*}
Therefore,
\begin{equation*}
\left(\frac{1}{4}-\frac{\delta^2}{4}\right)\int_{\de M} \phi\,d\sigma_1 \leq \left(\frac{1}{3\delta^2}-\frac{1}{4}\right)\int_{\de M}\phi\,d\sigma_2.
\end{equation*}
Choosing $\delta^2 = \frac{2}{3}$, we get the inequality
\begin{equation*}
\int_{\de M}\phi\,d\sigma_1 \leq 3 \int_{\de M}\phi\,d\sigma_2.
\end{equation*}
Equality follows from \eqref{im4}.
\end{proof}
Our final step will be to show that the support of the measure $\sigma$ is formed by critical points of $\D_n$. In this way, the assertion (2.3) of Theorem \ref{compactness} would be proved.

%
%
%
\begin{proposition}\label{p:tfgc12}For $n = 3$, let $(u_i)$ be a sequence of solutions of \eqref{tfgc1} with $I_i(u_i)$ uniformly bounded and $\mathscr{S}_0 \neq \emptyset$. Then, $\nabla^T \D_n = 0$ on  $\supp \sigma$. 
\end{proposition}
\begin{proof}
%
%
Integrating by parts in \eqref{pi3}, we obtain the following identity:
\begin{align} \label{tfgc13} \nonumber
\frac{1}{p_i+1}\int_M u_i^{p_i+1} F\cdot \nabla K_i +  \frac{8}{p_i+3}\int_{\de M} u_i^{\frac{p_i+3}{2}} F\cdot \nabla^T H_i + \frac{1}{p_i+1}\int_M K_i u_i^{p_i+1} \dive F \\ \nonumber +  \frac{8}{p_i+3} \int_{\de M}H_i u_i^{\frac{p_i+3}{2}} \dive^T F - 4\int_M \modgrad{u_i}\dive F + \frac{1}{2} \int_M S F\cdot \nabla (u_i^2)= \\ 
\frac{1}{p_i+1} \int_{\de M} K_i u_i^{p_i+1} F \cdot \eta - 8\int_M DF(\nabla u_i, \nabla u_i) - 4\int_{\de M}\modgrad{u_i}F\cdot \eta.
\end{align}

The general idea is to choose suitable vector fields $F$ in this identity, divide by $\rho_i$ and take the limit. Using H\"{o}lder's inequality, one can see that

\begin{equation*}  \rho_i^{-1} \int_M S\  F \cdot \nabla(u_i^2)  \to 0. \end{equation*}


As a first step we consider the distance function $\mathtt{d}$ from the boundary, and with 
positive sign inside $M$. Given a small number $\delta >0$, we consider smooth non-increasing cut-off function $\chi_\delta(t)$ such that 
$$
  \begin{cases}
  	\chi(t) = 1 & t \leq \delta; \\ 
  	\chi(t) = 0 & t \geq 2 \delta, 
  \end{cases}
$$
and the vector field $F = \mathtt{d} \chi(\mathtt{d}) \nabla \mathtt{d}$. It is well known that 
the divergence of $\nabla \mathtt{d}$ equals the opposite of the mean curvature of the 
level sets of $\mathtt{d}$, see e.g. \cite[Chapter 6: Exercise 11.d]{docarmo}, 
which are smooth near $\partial M$. Therefore we obtain that 
$$
  \dive F = 1 + O(\mathtt{d}); \quad F|_{\partial M} = 0. 
$$

By \eqref{tfgc13}, we obtain
\begin{align*} 
	\frac{1+o_i(1)}{\dst}\int_M \abs{K_i} u^{p_i+1}   + (4+o_i(1))\int_M \modgrad{u_i} =  (8+o_i(1))\int_M 
	(\nabla u_i \cdot \, \nabla \mathtt{d})^2. 
\end{align*}
Using Proposition \ref{im9} we then deduce  
\begin{equation} \label{tfgc15}
	\int_M \modgrad{u_i} = (1 + o_i(1)) \int_M (\nabla u_i \cdot \, \nabla \mathtt{d})^2, 
\end{equation}
which means that the gradient of $u_i$ is mostly normal to the boundary.

We now choose an arbitrary vector field $V$ tangent to $\partial M$, and extend it as a vector 
field $F$ in the interior of $M$ such that $\dive F = 0$ in a neighborhood of 
$\partial M$. Observe that now all terms in \eqref{tfgc13} involving  $F \cdot \nu$ on $\partial M$ or $\dive F$ will cancel.

If one splits on $\partial M$ a vector field $\tilde{F}$ into its tangential component $\tilde{F}^T$ 
and it normal one $\tilde{F}_\eta$,  for a tangent vector $v$ to $\partial M$ 
there holds $\nabla^M_v \tilde{F}^T = \nabla^{\partial M}_v \tilde{F}^T + A(v,\tilde{F}^T)\eta$,
where $A$ is the second fundamental form of $\partial M$. Taking the 
trace of this relation and adding the covariant derivative of the normal component, 
one finds for $\tilde{F} = F$ as above 
$$
  0 = (\dive_M F)|_{\partial M} = \dive^T F + h \langle \eta, F \rangle + D_{\eta} F_\eta, 
$$
where $h$ is the mean curvature of the boundary. By the fact that $\langle \eta, F \rangle = 0$ 
and $F = V$ on $\partial M$,  this implies in particular  
$$
  D_{\eta} F_\eta = - \dive^T V \mbox{ on } \partial M. 
$$
By \eqref{tfgc15} we then find 
$$
 - 8 \rho_i^{-1} \int_M DF(\nabla u_i, \nabla u_i)  \longrightarrow  \frac{1}{4} \int_{\partial M} \dive^T V \, d \sigma. 
$$

Observe also that, by Proposition  \ref{im9},

$$ \rho_i^{-1} \frac{8}{p_i+3} \int_{\de M} H_i u_i^{\frac{p_i+3}{2}} \dive^T F \to \frac{1}{4} \int_{\partial M} \dive^T V \, d \sigma. $$

Hence, when we divide \eqref{tfgc13} by $\rho_i$ and pass to the limit, the above two terms cancel and we obtain:

$$  \frac 14 \int_{\partial M} \frac{1}{H} \, (V \cdot \nabla^T H) \, d \sigma - \frac 18 
  \int_{\partial M} \frac{1}{K} \,  (V \cdot \,  \nabla^T K) \, d \sigma = 0.
$$
Notice however that, by Lemma \ref{tfgc11}, $\sigma$  is supported 
in $\{\D_n = 1 \}$, and hence on its support we have that $\frac{1}{H} \nabla^T H - \frac{1}{2K} \nabla^T K  
= \frac{1}{\D_n} \nabla^T \D_n = \nabla^T \D_n$. This observation and the 
last formula then imply 
$$
 \int_{\partial M} (V \cdot \, \nabla^T \D_n) \, d \sigma = 0 
$$
for all vector fields $V$ on $\partial M$. This means that $\nabla^T \D_n = 0$ on the support of $\sigma$, 
as desired. 
\end{proof}

\medskip

The proposition implies (2.3) in Theorem \ref{compactness}, since we were assuming that 
$1$ is a regular value of $\D_n$.

\medskip 

\begin{remark}\label{r:dim-D}
The proof of the latter proposition would work with minor changes in general dimension 
if we had anyway the conclusion of Lemma \ref{tfgc11} about the support of the measure $\sigma$.  	
\end{remark}

\section{Appendix: proof of Lemma \ref{tf15}} \label{s:app} \setcounter{equation}{0}

\begin{proof}  Since $M$ is a smooth manifold with boundary, we can find an extension $(\hat{M},\hat{g})$ of 
	the Riemannian structure including an exterior neighborhood of $\partial M$. 
	For $p\in\partial M$, let $\eta(p)$ be the  the exterior unit normal vector: 
	given two  two positive parameters $\beta$ and $D$, with $\beta$ small and $D > \frac{1}{\sqrt{n(n-1)}}$, 
	define the point 
	\begin{equation*}
		P_{\beta,D}= Exp^{\hat{g}}_p  (\sqrt{n(n-1)}D\beta\eta(p)), 
	\end{equation*}
where $Exp^{\hat{g}}_p $ stands for the exponential map of $\hat{g}$ at $p$. 
Inspired by the solutions to  problem \eqref{lp1} classified in \cite{ChipotFilaShafrir} , we consider the family of functions defined on $M$
	\begin{equation}
		\varphi_{\beta,D}(x) = \frac{\beta^{\frac{n-2}{2}}}{\left(\dist_{\hat{g}}{(x, P_{\beta,D})}^2-\beta^2\right)^{\frac{n-2}{2}}},
	\end{equation}
	and the modified family
	\begin{equation}\label{tf13}
		\func(x)=\mu^{\frac{n-2}{2}} \varphi_{\beta, D}(x),
	\end{equation}
	where $\mu$ is a positive constant, yet to be set. The strategy will be to estimate the energy $I$ on 
	such functions,  choosing an adequate value for $\mu$ so that $I(\func)\to -\infty$ as $D\to \frac{1}{\sqrt{n(n-1)}}$. In order to simplify the notation, let us also define
	\begin{equation*}
		{\e_D}^2 = n(n-1)D^2-1,
	\end{equation*}
	and notice that $\e_D\to 0$ as $D\to\frac{1}{\sqrt{n(n-1)}}.$ It is easy to see that
	\begin{align}\label{tf12}I(\func)&=\frac{2(n-1)}{n-2}\mu^{n-2}\int_M \modgrad{\varphi_{\beta,D}}+\frac{\mu^{n-2}}{2}\int_M S {\varphi_{\beta, D}}^2 + \nonumber \\&+\frac{n-2}{2n}\mu^n\int_M \abs{K}{\varphi_{\beta, D}}^\dst-(n-2)\mu^{n-1}\int_{\partial M}H {\varphi_{\beta,D}}^\dsh.
	\end{align}
	Take $r>0$ small but fixed: it is easy to show that all the above integrals, 
	restricted to the regions (both at the interior and at the boundary) outside $B(p,r)$ tend 
	to zero as $\beta \to 0$, uniformly for $D$ close to $\frac{1}{\sqrt{n(n-1)}}$. 
	
	Having observed this, we proceed by estimating the \emph {boundary term} as 
	\begin{align*}
		\mbox{(B)}:=\int_{\partial M}H{\varphi_{\beta, D}}^\dsh &= \int_{\partial M\cap B^{n-1}(p,r)}H{\varphi_{\beta, D}}^\dsh+
		o_\beta(1) \\ &\geq \left(\min_{\partial M \cap B^{n-1}(p,r)}H\right) \int_{\partial M\cap B^{n-1}(p,r)}{\varphi_{\beta, D}}^\dsh + o_\beta(1), 
	\end{align*}
 where $o_\beta(1)$ tends to zero as $\beta$ does. 
Take normal coordinates $\{x_1,\ldots,x_n\}$ around $P_{\beta,D}$ in such a way that the 
geodesic joining $p$ to $P_{\beta,D}$ is mapped into the $x_n$-axis, and set $x' = (x_1, \dots, x_{n-1})$. 
We can rewrite the previous integral in the following way, up to multiplicative errors of the form $(1 + o_r(1))$,
with $o_r(1)$ tending to zero as $r \to 0$,  due to the presence of integral volume elements: 
	\begin{align*}
		\mbox{(B)} &\geq \left(\min_{\partial M \cap B^{n-1}(p,r)}H\right) \int_{B^{n-1}(p,r)}\frac{\beta^{n-1}dx_1\cdots dx_{n-1}}{\left(|x'|^2+n(n-1)D^2\beta^2-\beta^2\right)^{n-1}} + o_\beta(1)\\ &= \left(\min_{\partial M \cap B^{n-1}(p,r)}H\right) \int_0^r\int_{\Sph^{n-2}(0,s)}\frac{\beta^{n-1}dx_1\cdots dx_{n-1}ds}{\left(s^2+\beta^2{\e_D}^2\right)^{n-1}} + o_\beta(1) \\ &= \left(\min_{\partial M \cap B^{n-1}(p,r)}H\right) \vert \Sph^{n-2}\vert \beta^{n-1}\int_0^r\frac{s^{n-2}ds}{(s^2+\beta^2{\e_D}^2)^{n-1}}
		+ o_\beta(1). 
	\end{align*}
	The latter integral can be addressed via the change of variable $t=\frac{s}{\e_D\beta}$:
	\begin{equation}\label{tf3}
		\beta^{n-1}\int_0^r\frac{s^{n-2}ds}{(s^2+\beta^2{\e_D}^2)^{n-1}} = \frac{1}{{\e_D}^{n-1}}\int_0^{\frac{r}{\e_D\beta}}\frac{t^{n-2}dt}{(1+t^2)^{n-1}}.
	\end{equation}
	We can always take $r,\beta$ and $D$ in such a way that $\frac{r}{\e_D \beta}\to +\infty$ as $\e_D$ tends to zero. Calling
	\begin{equation*}
		\gamma_n = \int_0^{+\infty}\frac{t^{n-2}dt}{(1+t^2)^{n-1}} = \frac{\sqrt{\pi}\:\Gamma\left(\frac{n-1}{2}\right)}{2^{n-1}\Gamma\left(\frac{n}{2}\right)},
	\end{equation*}
	we finally have:
	\begin{equation}\label{tf1}
		\int_{\partial M} H {\varphi_{\beta, D}}^\dsh \geq H(p) \vert \Sph^{n-2}\vert \gamma_n\frac{1+o_r(1)}{{\e_D}^{n-1}}+o_\beta(1).
	\end{equation}
	Now, in \eqref{tf12} we bound the \emph{critical term in the interior} of $M$ as:
	\begin{align*}
		\mbox{(I)}:&= \int_M \abs{K}{\varphi_{\beta, D}}^\dst = \int_{M\cap B^{n}(p,r)}\abs{K}{\varphi_{\beta, D}}^\dst  
		+ o_\beta(1) \\ &\leq \left(\max_{M\cap B^n(p,r)}\abs{K}\right)\int_{M\cap B^{n}(p,r)}{\varphi_{\beta, D}}^\dst + o_\beta(1). 
	\end{align*}
	Taking normal coordinates  as before we obtain, up to multiplicative constants of 
	order $(1 + o_r(1))$:
	\begin{align*}
		\mbox{(I)}\leq \left(\max_{M\cap B^n(p,r)}\abs{K}\right) \int_{B^n_+(0,r)}\frac{\beta^n dx_1\cdots dx_n}{\left(|x'|^2+(x_n+\sqrt{n(n-1)}D\beta)^2-\beta^2\right)^n} + o_\beta(1). 
	\end{align*}
	Notice that
	\begin{align*}
		\mbox{(I')}&:=\int_{B^n_+(0,r)}\frac{\beta^n dx_1\cdots dx_n}{\left(|x'|^2+(x_n+\sqrt{n(n-1)}D\beta)^2-\beta^2\right)^n} \\ &=\int_0^r\int_0^{\sqrt{r^2-x_n^2}}\int_{\Sph^{n-2}}\frac{\beta^n dx_1\cdots dx_nds}{\left(s^2+(x_n+\sqrt{n(n-1)}D\beta)^2-\beta^2\right)^n} \\ &\leq \abs{\Sph^{n-2}}\beta^n \int_0^r\int_0^r\frac{s^{n-2} dx_nds}{\left(s^2+x_n^2+{\e_D}^2\beta^2+2x_n\sqrt{n(n-1)}D\beta\right)^n}.
	\end{align*}
	Since we are taking $r\to 0$, $x_n^2$ is negligible compared to $x_n$. This fact, together with Fubini's Theorem, gives:
	\begin{align*}
		\mbox{(I')} &\leq \abs{\Sph^{n-2}}\beta^n \int_0^r\left(\int_0^r\frac{s^{n-2}dx_n}{\left(s^2+{\e_D}^2\beta^2+2x_n\sqrt{n(n-1)}D\beta\right)^n}\right)ds =\\ &= \frac{\abs{\Sph^{n-2}}\beta^{n-1}}{2(n-1) D \sqrt{n(n-1)}} \biggl\{ \int_0^r \frac{-s^{n-2}ds}{\left(s^2+2r\sqrt{n(n-1)}D\beta+{\e_D}^2\beta^2\right)^{n-1}}+\\ &+\int_0^r\frac{s^{n-2}ds}{\left(s^2+{\e_D}^2\beta^2\right)^{n-1}} \biggr\}.
	\end{align*}
	Since $\beta$ and $r$ are small but fixed parameters, the first term of the sum is non-singular and can be uniformly bounded. On the other hand, the second term has been computed in \eqref{tf3}. If we put everthing together, we find:
	\begin{equation}\label{tf4}
		\int_M \abs{K}{\varphi_{\beta, D}}^\dst \leq \abs{K(p)}\gamma_n \frac{\abs{\Sph^{n-2}}}{2(n-1)}\frac{1+o_r(1)}{{\e_D}^{n-1}}
		+ o_\beta(1). 
	\end{equation}
	We can repeat these computations for the \emph{quadratic term} in \eqref{tf12}, and eventually obtain:
	\begin{equation}\label{tf5}
		\int_M S{\varphi_{\beta, D}}^2 \leq S(p) \frac{\abs{\Sph^{n-2}}\Gamma\left(\frac{n-5}{2}\right)\Gamma\left(\frac{n-1}{2}\right)}{4(n-3)\beta^2\Gamma\left(n-3\right)}\frac{1+o_r(1)}{{\e_D}^{n-5}}+o_\beta(1) = o\left(\inv{{\e_D}^{n-2}}\right).
	\end{equation}
	Finally, we address the \emph{Dirichlet energy}. A direct computation shows  that
	\begin{equation*}
		\abs{\nabla \varphi_{\beta, D}(x)}= \frac{n-2}{2}\beta^{\frac{n-2}{2}}\frac{ \abs{\nabla\dist(x,P_{\beta,D})^2}}{\left(\dist(x,P_{\beta,D})^2-\beta^2\right)^\frac{n}{2}}.
	\end{equation*}
	Using the inequality $\abs{\nabla\dist(x,p)^2}\leq 2\dist(x,p)$, we have
	\begin{equation}\label{tf7}
		\abs{\nabla \varphi_{\beta, D}(x)} \leq (n-2)\beta^{\frac{n-2}{2}}\frac{ \dist(x,P_{\beta,D})}{\left(\dist(x,P_{\beta,D})^2-\beta^2\right)^\frac{n}{2}}.
	\end{equation}
	Reasoning as before, we find 
	\begin{equation}\label{tf10}
		\int_M \modgrad{\varphi_{\beta, D}}\leq \int_{M\cap B^{n}(p,r)}\modgrad{\varphi_{\beta, D}} + o_\beta(1). 
	\end{equation}
	Again, taking normal coordinates as above we have, up to multiplicative errors of the form $(1+o_r(1))$
	\begin{align}\label{tf9}
		&\mbox{(D')}:= \int_{M\cap B^{n}(p,r)}\modgrad{\varphi_{\beta, D}} =\nonumber\\ &= (n-2)^2 \beta^{n-2}\int_{B^n_+(0,r)}{\frac{\left(|x'|^2+\left(x_n+\sqrt{n(n-1)}D\beta\right)^2\right)dx}{\left(|x'|^2+\left(x_n+\sqrt{n(n-1)}D\beta\right)^2-\beta^2\right)^n}} \nonumber\\ &= (n-2)^2\left\lbrace \int_{B^n_+(0,r)}\frac{\beta^{n-2}dx_1\cdots dx_n}{\left(|x'|^2+\left(x_n+\sqrt{n(n-1)}D\beta\right)^2-\beta^2\right)^{n-1}}\right. \nonumber\\ &+ \left. \int_{B^n_+(0,r)}\frac{\beta^{n}dx_1\cdots dx_n}{\left(|x'|^2+\left(x_n+\sqrt{n(n-1)}D\beta\right)^2-\beta^2\right)^{n}}\right\rbrace.
	\end{align}
	The first integral term is of lower order $o\left(\frac1{{\e_D}^{n-2}}\right)$, and the second one is exactly (I'), which is already calculated. Finally, combining \eqref{tf10} and \eqref{tf9}, we conclude:
	\begin{equation}\label{tf11}
		\int_M \modgrad{\varphi_{\beta, D}} \leq \gamma_n\frac{(n-2)^2\abs{\Sph^{n-2}}}{2(n-1)}\frac{1+o_r(1)}{{\e_D}^{n-1}}+o\left(\inv{{\e_D}^{n-2}}\right).
	\end{equation}
	Recalling that $\func(x)=\mu^{\frac{n-2}{2}} \varphi_{\beta, D}(x),$ and substituting \eqref{tf1}, \eqref{tf4}, \eqref{tf5} and \eqref{tf11} into \eqref{tf12} we finally have, up to 
	multiplicative contants of the form $(1 + o_r(1))$ in each term: 
	\begin{align}\label{tf14}
		I(\func) &\leq \gamma_n \mu^{n-2}\abs{\Sph^{n-2}}(n-2)\left( \frac{\abs{K(p)}}{4n(n-1)}\mu^2-H(p)\mu+1\right)\inv{{\e_D}^{n-1}} +\nonumber\\ &+o\left(\inv{{\e_D}^{n-2}}\right).
	\end{align}
	Denoting by $P$ the polynomial in $\mu$
	\begin{equation}\label{tf16}
		P(\mu) =  \frac{\abs{K(p)}}{4n(n-1)}\mu^2-H(p)\mu+1,
	\end{equation}
	we know that there exist values  for which $P(\mu)<0$ if, and only if, its discriminant is positive, that is,
	\begin{equation}
		H(p)^2- \frac{\abs{K(p)}}{n(n-1)} > 0,
	\end{equation}
	which is equivalent to our hypothesis $\D_n(p)>1$. In view of \eqref{tf14}, we can assert that there exist $\beta, r >0$ 9small) and $\mu >0$ such that $I(\func) \to -\infty$ as $D\to \inv{\sqrt{n(n-1)}}$.
\end{proof}
\begin{remark}
	Minimizing in \eqref{tf16} we can compute the optimal value for the constant, namely,
	\begin{equation*}
		\mu = \frac{2n(n-1)H(p)}{\abs{K(p)}} = 2\frac{\D_n(p)}{\sqrt{\abs{K(p)}}}.
	\end{equation*}
\end{remark}

\end{document}